\documentclass[12pt,a4paper,twoside]{article}
\usepackage{amssymb}
\usepackage{amsfonts}
\usepackage{geometry}
\usepackage{color}
\usepackage{mathtools}
\usepackage{amsmath,amsthm,amssymb,amsfonts,enumitem}
\usepackage{CJK}
\usepackage{latexsym}
\usepackage{graphicx}
\usepackage{pgfpages}
\usepackage{mathrsfs}
\usepackage{ytableau}
\usepackage{enumerate}
\usepackage{cite}
\usepackage[hyphens]{url}
\usepackage[bookmarks=false,hidelinks]{hyperref}
\usepackage[affil-it]{authblk}
\usepackage{appendix}
\usepackage{amsmath}
\usepackage{comment}

\newtheorem{definition}{Definition}[section]
\newtheorem{theorem}[definition]{Theorem}
\newtheorem{lemma}[definition]{Lemma}
\newtheorem{proposition}[definition]{Proposition}
\newtheorem{corollary}[definition]{Corollary}
\newtheorem{remark}[definition]{Remark}

\numberwithin{equation}{section}
\title{Exponential Mixing for 2D Stochastic Damped Euler Equation Driven by Bounded Noise}
\author[a]{Rui Bai}
	\author [a] {Chunrong Feng}
	\author[a,b]{Huaizhong Zhao}
	\affil[a]{Department of Mathematical Sciences, Durham University, DH1 3LE, UK}
	\affil[b] {School of Mathematics, Shandong University, Jinan 250100, China}
	
	\affil[ ]{rui.bai@durham.ac.uk, chunrong.feng@durham.ac.uk,  huaizhong.zhao@durham.ac.uk}
	\date{}

\begin{document}

\maketitle
\begin{abstract}
    In this paper, we study the long-time behaviour of the two-dimensional stochastic damped Euler equation on the torus driven by bounded random forcing. Unlike stochastic Navier–Stokes or fractionally dissipative Euler equations, the model possesses no viscous regularisation, so the classical parabolic smoothing is unavailable. We prove that when the damping coefficient is sufficiently large, the associated Markov semigroup admits a unique invariant measure and converges exponentially fast to equilibrium.  The key ingredient is the establishment of a global-in-time $W^{1,\infty}$ estimate for the vorticity. This estimate yields a compact absorbing mechanism in $C(\mathbb {T}^2)$, which enables us to establish the uniqueness of the invariant measure and exponential mixing. To the best of our knowledge, this is the first exponential mixing result for a genuinely inviscid stochastic Euler-type equation. Our approach demonstrates that sufficiently strong linear damping can effectively replace the compactness mechanism usually provided by viscosity and is expected to be applicable to other inviscid or weakly dissipative stochastic partial differential equations driven by bounded random forcing.
     \vskip10pt
    
    \noindent
{\bf MSC2020 subject classifications:} Primary 60H15, 37A25; secondary 35Q31, 35R60\vskip2pt

    \end{abstract}
    
\section{Introduction}

The 2D stochastic Euler equation has been extensively studied in recent years as a fundamental model for understanding the interaction between randomness and inviscid fluid dynamics. In contrast to the stochastic Navier–Stokes equations or the stochastic fractional Euler equations, the stochastic Euler equation possesses no viscous regularisation; therefore, many analytical techniques based on parabolic smoothing no longer work. The resulting difficulties have stimulated considerable interest in the existence, regularity, and long-time behaviour of stochastic Euler flows. 
The fundamental well-posedness results for stochastic Euler equations have been established by many works, including Bessaih-Flandoli\cite{bessaih1999}, Brzeźniak-Peszat\cite{brzezniak2001stochastic}, and Lang-Crisan\cite{lang2023well}. Regarding the long-time behaviour of stochastic Euler equations,
several important results have been obtained. The unique ergodicity result for the fractionally dissipated stochastic Euler equation has been established in Constantin, Glatt-Holtz and Vicol\cite{constantin2014unique}. As for the long-time behaviour of stochastic Euler equations with linear damping, Bessaih established the existence of weak random attractors \cite{Bessaih2000stochastic} and stationary solutions \cite{bessaih2008stationary}, while Bessaih and Ferrario proved the existence of invariant measures \cite{bessaih2020invariant}. To the best of our knowledge, no uniqueness or exponential mixing result is currently available in this setting. In this paper, we will establish the exponential mixing result for the stochastic Euler equation with a large linear damping term, driven by bounded random forcing.

From the modelling perspective, bounded random forcing arises naturally in many physical situations where the amplitude of external fluctuations is constrained. In realistic fluid systems, environmental forcing, control inputs, or unresolved small-scale effects are often subject to finite energy budgets and cannot attain arbitrarily large values. In such circumstances, bounded random perturbations provide a more realistic description than Gaussian forcing, whose trajectories are unbounded with positive probability. 
The study of bounded noise models has attracted considerable interest in the theory of stochastic dynamical systems  Kuksin-Piatniski-Shirikyan\cite{kuksin2001coupling}, Kuksin-Shirikyan\cite{kuksin2002coupling,kuksin2012mathematics}. This programme established the uniqueness of stationary measures and exponential mixing for a broad class of stochastic PDEs Kuksin-Piatniski-Shirikyan\cite{kuksin2020exponential1}, Kuksin-Zhang\cite{kuksin2020exponential2}. A distinctive feature of this theory is that it does not rely on Gaussian forcing; instead, it treats bounded random perturbations as deterministic dynamical inputs and exploits the stability and compactness properties of the underlying equations.

Most existing results in this direction concern dissipative systems possessing a strong regularising mechanism. Typical examples include reaction–diffusion equations, complex Ginzburg–Landau equations, Navier–Stokes equations, and other parabolic models \cite{kuksin2020exponential1,kuksin2020exponential2}. In all these cases, compactness is ultimately generated by viscous or parabolic smoothing. Such regularisation plays a fundamental role in the construction of invariant measures and in the analysis of convergence towards equilibrium.
The situation is considerably different for inviscid fluid equations. In the absence of viscosity, there is no parabolic regularisation and no obvious compactness mechanism. Consequently, many of the techniques developed for dissipative stochastic PDEs are no longer available. In the context of fluid equations, bounded random forcing is especially appealing because it allows one to investigate whether turbulent or statistically stationary behaviour can emerge from physically constrained perturbations. Understanding whether bounded random forcing can still generate robust ergodic behaviour for inviscid fluid models remains a largely open problem.

 The combination of bounded forcing and the absence of viscosity makes the stochastic damped Euler equation a particularly challenging case for the bounded-noise theory of infinite-dimensional dynamical systems.
In this paper, we study the stochastic damped Euler equation driven by bounded noise.
These equations are given by 
\begin{equation}\label{eqn:Euler 0}
\begin{cases}
    \partial_t u+\gamma u(t)+(u\cdot \nabla )u(t)+ \nabla p =   H(t),\\
    \nabla\cdot u=0.\end{cases} 
\end{equation}
Here, the damping coefficient $\gamma$ is assumed to be positive. The unknown variables are the velocity vector $u=u(t,\xi)$ and the pressure $p=p(t,\xi)$, where $t$ is the time variable and $\xi \in \mathbb{T}^2$ represents the spatial variable, $H(\cdot)$ is a bounded random process in a function space. We define the vorticity $w$ by $w = \nabla\wedge u = \partial_1u_2- \partial_2u_1$. Note that $u$ can be
recovered from $w$ and the condition $\nabla \cdot u=0$. We study the vorticity form of the above equations
\begin{equation}\label{eqn:Euler 1}
\begin{cases}
     \partial_t w+\gamma w(t) +(\mathcal{K}\ast w)(t)\cdot \nabla w(t) =  \eta(t),\\
     u= \mathcal{K}\ast w,\; \nabla \cdot u=0,
    \end{cases}
\end{equation}
where $\mathcal{K}$ is the Biot-Savart kernel, $\eta=\nabla \wedge H $.

 Although the damping term contributes to dissipation at the level of
energy, it does not provide any direct smoothing of spatial
oscillations compared to viscous or fractional
dissipation. Our objective is to understand whether the damping term alone is sufficient to generate compactness and ergodicity. Our main result states that when the damping coefficient is sufficiently large, the associated Markov semigroup admits a unique invariant measure and converges exponentially fast towards equilibrium. 
The main contribution of this paper is the discovery that
sufficiently strong damping generates a global
\(W^{1,\infty}\)-absorbing structure for the stochastic dynamics. We prove that when the damping coefficient is sufficiently large, solutions satisfy a global-in-time estimate in $W^{1,\infty}$,
$$\sup_{t\ge0}\|w(t)\Vert_{W^{1,\infty}}<\infty .$$
This estimate is critical in establishing a compact absorbing structure for the stochastic dynamics and an exponential mixing result. To the best of our knowledge, this mechanism has not previously been exploited in the study of invariant measures and exponential mixing for stochastic Euler equations. Unlike the classical bounded-noise theory, where compactness is produced by parabolic smoothing, here compactness emerges entirely from the damping mechanism. In this sense, strong damping effectively replaces the role usually played by viscosity in the study of long-time dynamics.


Our results should be contrasted with two previous directions. In the Gaussian forcing setting, important contributions were made by Hairer and Mattingly \cite{hairer2006ergodicity},   Constantin, Glatt-Holtz and Vicol \cite{constantin2014unique}, which established the existence of invariant measures and proved unique ergodicity for the two-dimensional stochastic fluid equation with a viscous Laplacian or fractional Laplacian. 
A crucial ingredient in their analysis is the presence of the (fractional) Laplacian. Another important contribution was made by Bessaih and Ferrario \cite{bessaih2020invariant}, who studied two-dimensional stochastic Euler equations and proved the existence of invariant measures in the weak$\star$ topology of $L^\infty$. Their work represents one of the first rigorous constructions of invariant measures for stochastic Euler equations and demonstrates that the vorticity formulation provides a natural framework for the analysis of long-time behaviour. 
The present work differs from these previous results in several significant aspects. 

First, we consider bounded random forcing rather than Gaussian forcing. This places the problem closer to the bounded-noise stochastic dynamics framework developed by Kuksin and Shirikyan and removes many probabilistic tools available in the Gaussian setting. Second, our equation contains no Laplacian and no fractional dissipation. Consequently, the compactness mechanisms available in dissipative fluid models are absent. Third, 
we establish a global-in-time estimate in $W^{1,\infty}$. This allows us to construct invariant measures supported on a substantially more regular phase space $C(\mathbb{T}^2)$ and provides considerably stronger compactness properties for the dynamics. 

Most importantly, our results go beyond the existing theory currently available for stochastic Euler equations. We prove the uniqueness of invariant measures and exponential mixing when the damping coefficient is sufficiently large. In particular, for any $x \in C(\mathbb{T}^2)$ and $\phi \in \text{Lip}_b(C(\mathbb{T}^2))$,
    \begin{equation}
        |P_i\phi(x)-\int_{E}\phi\;d\mu\vert \leq c(x)e^{-c(\gamma)i}\|\phi\Vert_{\text{Lip}},
    \end{equation}
where $c(x)$, $c(\gamma)>0$ depend on $x$ and $\gamma$ respectively. Thus, the main contribution of the paper is not merely the construction of invariant measures but also the identification of a mechanism through which strong damping generates a global $W^{1,\infty}$-absorbing structure and leads to a complete ergodic description of an inviscid fluid equation driven by bounded noise.

To the best of our knowledge, this is the first exponential mixing result for a genuinely inviscid stochastic Euler-type equation driven by bounded noise. The key novelty is that strong damping compensates for the
complete absence of differential dissipation, as in the case of stochastic Navier-Stokes equations, where the dissipation essentially increases arbitrarily large for high modes. 
From a broader perspective, this work may be viewed as a first step towards extending the bounded-noise stochastic dynamics programme of Kuksin and Shirikyan to inviscid fluid equations. We expect that the mechanism developed here will be useful for other inviscid or weakly dissipative stochastic PDEs driven by bounded random forcing.

\section{Preliminaries}\label{Sec:Pre}
\subsection{Setting and assumptions}
We consider equation (\ref{eqn:Euler 0}) with periodic boundary conditions. Denoting $\mathbb{T}^2:=[0,2\pi]^2$, we define 
$$H:=\{g\in [L^2(\mathbb{T}^2)]^2:\int_{\mathbb{T}^2} g(\xi)\;d\xi=0,\; \nabla\cdot g=0,\;g\; \text{is periodic in} \; \mathbb{T}^2\}.$$ It is well-known that $H$ is a Hilbert space when endowed with the standard $[L^2(\mathbb{T}^2)]^2$ inner
product. We define $V=[H^1(\mathbb{T}^2)]^2\cap H$ and $V^{1,p}= [W^{1,p}(\mathbb{T}^2)]^2\cap H $, where $H^k(\mathbb{T}^2)$ is denoted as the Sobolev space $W^{k,2}(\mathbb{T}^2)$.

We define the tri-linear form $b:V\times V\times V\mapsto \mathbb{R}$ by 
$$b(u,v,w)= \int_{\mathbb{T}^2}(u(\xi)\cdot \nabla)v(\xi)\cdot w(\xi) \;d\xi, u,v,w \in V.$$ It is well known that $b$ is well-defined and continuous on $V\times V\times V$. Let us denote by $V^\star$ the dual space of $V$, then $b$ induces the continuous mappings $B:V\times V \mapsto V^\star $  defined by
$$\langle B(u,v),w\rangle_{V^{\star},V}:= b(u,v,w) ,$$ for $u,v,w \in V$. By the incompressibility condition,
for $u,v,w \in V$, we have 
$$\langle B(u,v),w\rangle_{V^{\star},V}=- \langle B(u,w),v\rangle_{V^{\star},V}, $$ which implies that $$ \langle B(u,v),v\rangle_{V^{\star},V}=0. $$

 We define $\mathcal{H}:=L_0^2$ to be the space of real-valued square-integrable functions on the torus with a vanishing mean and define $\mathbb{Z}^2_0:= \mathbb{Z}^2\setminus \{(0,0)\} $. Then, 
$\{e_k\}_{k \in \mathbb{Z}^2_0 } $ defined by
$$e_k(\xi)=\frac{1}{2\pi}e^{i\xi\cdot k},\;k=(k_1,k_2) \in \mathbb{Z}^2_0, \;\xi \in \mathbb{T}^2$$ forms a complete orthogonal system of a complex space $\mathcal{H}_{\mathbb{C}}$ that is the complexification of $\mathcal{H}$. Moreover, we have that 
$$(\mathcal{K}\ast e_k)(\xi) = \frac{i}{2\pi}\frac{(k_2,-k_1)}{k_1^2+k_2^2}e^{i\xi \cdot k}.$$ Now, we use Fourier modes to describe noise $\eta$. Let $(\Omega, \mathcal{F},\mathbb{P})$ be a complete space, and let $\{\eta_k(t),t\geq 0\}_{k \in \mathbb{Z}^2_0 }$ be a sequence of real valued $\mathbb{P}-a.s.$ stochastic processes defined on it. We suppose that $\{\eta_k(t),t\geq 0\}_{k \in \mathbb{Z}^2_0 } $ are distributed as a certain process $\eta_0$, and $\eta_0|_{[i-1,i)},\; i \in \mathbb{N}$ are i.i.d. So it is sufficient to consider $\eta_0|_{[0,1]} $. We assume $\eta_0|_{[0,1]} $ are $\mathbb{P}-a.s.$ bounded in $L^2[0,1]$. Then, the noise $\eta$ in equation (\ref{eqn:Euler 1}) is defined by 
\begin{equation}\label{eqn:noise}
    \eta(t,\xi):=\sum_{k \in \mathbb{Z}^2_0 }\eta_k(t)\mathrm{Re}(b_ke_k(\xi)),
    \end{equation}
where $b_k \in \mathbb{C}$. Throughout this paper, we assume that there exist $a>3$ and a deterministic
constant $B>0$ such that
\begin{equation}\label{eqn:noise H norm}
   \int_0^1 \|\eta(t)\|_{H^a}^2\,dt
\le
C\sum_{k\in\mathbb Z_0^2}|b_k|^2|k|^{2a}
\le
B^2,
\qquad \mathbb P\text{-a.s.}.
    \end{equation}

We state some basic known results
on the existence and uniqueness of the solutions of the 2D Euler equation under the above assumption. The following result is according to  \cite{bessaih1999} and \cite{bessaih2015stochastic}. In particular, this result works for bounded noise with some regularity assumptions.
\begin{theorem}\label{existence and uniqueness}
    Let $\gamma\geq 0$ and assume (\ref{eqn:noise H norm}) with $a >3$. \\
 i) If $u_0 \in V$ , then on each interval $[0,T]$ there exists at least one weak global solution for equation (\ref{eqn:Euler 0}) with the initial condition $u(0)=u_0$, in the sense that $\mathbb{P}$-a.s.
    $$u \in C([0,T];H)\cap L^2(0,T;V)$$ 
    and for every $ \phi \in V$, $t \in [0,T]$
    \begin{align*}
        \langle u(t) , \phi\rangle+ \gamma\int_0^t\langle u(s), \phi\rangle\;ds - \int_0^t\langle [u(s)\cdot \nabla] \phi, u(s)\rangle\;ds  = \langle u_0,\phi\rangle+ \int_0^t\langle H(s),\phi\rangle\;ds     \end{align*}
$\mathbb{P}$-a.s.. \\
ii) Let $p \in [2,\infty)$. If $u_0 \in V^{1,p} $, the weak global solution u satisfies $$u \in L^\infty(0,T;V^{1,p}), \;\mathbb{P}-a.s. . $$ \\
iii) If $u_0\in V$ and $w_0=\nabla \wedge u_0\in L^\infty$, then $w=\nabla\wedge u$ satisfies $$w \in L^\infty(0,T;L^\infty), \;\mathbb{P}-a.s.$$ and pathwise uniqueness holds.
\end{theorem}
In \cite{bessaih2020invariant}, the authors proved that, under suitable assumptions for the noise, the semigroup $P_t$ defined as
$$P_t\phi(x)= \mathbb{E}[\phi(w(t;x))], $$ is Markovian. Then, they proved that there exists at least one invariant measure on $\mathcal{B}(L^\infty,\mathcal{T}_{bw^\star})$, where $(L^\infty,\mathcal{T}_{bw^\star}) $ is the bounded weak$\star$ topology of $L^\infty$. In this paper, we further prove that under our assumptions on the noise $\eta$, if $w_0 \in W^{1,\infty}(\mathbb{T}^2)$, then $w \in L^\infty(0,T;W^{1,\infty}(\mathbb{T}^2))$. Moreover, if the damping term $\gamma$ is sufficiently large, then there exists a unique invariant measure for $P_t$ on $\mathcal{B}(C(\mathbb{T}^2))$. Before doing that, we need to introduce the $W^{1,\infty}(\mathbb{T}^2)$ space.

\subsection{$W^{1,\infty}(\mathbb{T}^2)$ space and its topology}
The Sobolev space $W^{-1,1}(\mathbb{T}^2)$ is defined as 
\begin{equation}\label{def: W 1, -1}
    W^{-1,1}(\mathbb{T}^2):=\left\{h_0+\sum_{i=1}^2 \partial_i h_i \in \mathcal{D}^\prime(\mathbb{T}^2):h_0,h_1,h_2 \in L^1(\mathbb{T}^2)  \right\},
    \end{equation}
    equipped with the norm 
    $$\|h\Vert_{W^{-1,1}(\mathbb{T}^2)}=\inf\left\{\|h_0\Vert_{L^1(\mathbb{T}^2)} + \sum_{i=1}^2 \|h_i\Vert_{L^1(\mathbb{T}^2)}: h= h_0+\sum_{i=1}^2 \partial_ih_i \right\}.$$ 
    By this definition, $C^\infty(\mathbb{T}^2)$ is dense in $W^{-1,1}(\mathbb{T}^2)$,  since $C^\infty(\mathbb{T}^2)$ is dense in $L^1(\mathbb{T}^2)$. Therefore, $W^{-1,1}(\mathbb{T}^2) $ is a separable Banach space.

    \begin{lemma} The dual space of $W^{-1,1}(\mathbb T^2) $ is $W^{1,\infty}(\mathbb T^2) $, i.e.,
        \begin{equation}
\bigl(W^{-1,1}(\mathbb T^2)\bigr)^*
\simeq
W^{1,\infty}(\mathbb T^2).          
\end{equation}
    \end{lemma}
\begin{proof}
Let
\[
A:L^1(\mathbb T^2)^3\rightarrow W^{-1,1}(\mathbb T^2),
\qquad
A(h_0,h_1,h_2)
=
h_0+\partial_1 h_1+\partial_2 h_2.
\]
By the definition of \(W^{-1,1}(\mathbb T^2)\), the operator \(A\) is surjective and
\[
W^{-1,1}(\mathbb T^2)
\simeq
L^1(\mathbb T^2)^3/\ker A,
\]
equipped with the quotient norm. Hence,
\[
\bigl(W^{-1,1}(\mathbb T^2)\bigr)^*
\simeq
\Bigl\{
G\in \bigl(L^1(\mathbb T^2)^3\bigr)^*
:
G|_{\ker A}=0
\Bigr\}.
\]
Since
\[
\bigl(L^1(\mathbb T^2)^3\bigr)^*
=
L^\infty(\mathbb T^2)^3,
\]
for every \(G\in (W^{-1,1}(\mathbb T^2))^*\), there exist
\(g_0,g_1,g_2\in L^\infty(\mathbb T^2)\) such that
\[
G(h_0,h_1,h_2)
=
\int_{\mathbb T^2} g_0 h_0\,d\xi
+
\int_{\mathbb T^2} g_1 h_1\,d\xi
+
\int_{\mathbb T^2} g_2 h_2\,d\xi .
\]
The condition \(G|_{\ker A}=0\) means that whenever
\[
h_0+\partial_1 h_1+\partial_2 h_2=0
\]
in the sense of distributions, one has
\[
\int_{\mathbb T^2} g_0 h_0\,d\xi
+
\int_{\mathbb T^2} g_1 h_1\,d\xi
+
\int_{\mathbb T^2} g_2 h_2\,d\xi
=
0.
\]
Taking
\[
h_0=-\partial_1\varphi,
\qquad
h_1=\varphi,
\qquad
h_2=0,
\]
with \(\varphi\in C^\infty(\mathbb T^2)\), we obtain
\[
-\int_{\mathbb T^2} g_0\,\partial_1\varphi\,d\xi
+
\int_{\mathbb T^2} g_1\,\varphi\,d\xi
=
0.
\]
Hence,
\[
g_1=-\partial_1 g_0
\]
in the sense of distributions. Similarly,
\[
g_2=-\partial_2 g_0.
\]
Since \(g_1,g_2\in L^\infty(\mathbb T^2)\), it follows that
\[
g_0\in W^{1,\infty}(\mathbb T^2).
\]
Therefore, for any
\[
h=h_0+\partial_1 h_1+\partial_2 h_2
\in W^{-1,1}(\mathbb T^2),
\]
the functional \(G\) can be written as
\[
G(h)
=
\int_{\mathbb T^2} g_0 h_0\,d\xi
-
\sum_{i=1}^{2}
\int_{\mathbb T^2}
\partial_i g_0\, h_i\,d\xi.
\]
This coincides with the natural pairing between
\(W^{1,\infty}(\mathbb T^2)\) and \(W^{-1,1}(\mathbb T^2)\). Consequently,
\[
\bigl(W^{-1,1}(\mathbb T^2)\bigr)^*
\simeq
W^{1,\infty}(\mathbb T^2).
\]
\end{proof}

    To shorten notation, we write $W^{1,\infty}$ for the space $W^{1,\infty}(\mathbb{T}^2 ) $ and $W^{-1,1} $ for the space $W^{-1,1}(\mathbb{T}^2) $. By the discussion above, $W^{1,\infty} $ is the dual space of $W^{-1,1} $. The space $W^{1,\infty} $ is not separable under norm topology, whereas the space $W^{-1,1} $ is separable. In fact, $C^\infty(\mathbb{T}^2) $ is separable, and $C^\infty(\mathbb{T}^2)$ is dense in $W^{-1,1}$ by the above definition. We recall the meaning of convergence in $W^{1,\infty} $ with respect to weak$\star$ topology: $x_n\overset{\star} {\rightharpoonup} x$ in $W^{1,\infty} $ means
    $$\langle x_n,h \rangle_{W^{1,\infty},W^{-1,1}} \rightarrow \langle x,h \rangle_{W^{1,\infty},W^{-1,1}},\; \forall h \in W^{-1,1}. $$ Since $ W^{1,\infty}$ can be characterised by Lipschitz continuous functions, by the Lipschitz functions approximation Theorem (see Theorem 6.11 of \cite{MR3409135}), for any $x \in W^{1,\infty} $ and $\varepsilon>0 $, there exist $N>0$ and 
    $x_n \in C^1$ such that $$\mathcal{L}^2\left(\left\{\xi: x(\xi)\neq x_n(\xi)\; or\; \nabla x(\xi) \neq \nabla x_n(\xi) \right\}\right)<\varepsilon ,\; \forall n \geq N, $$ where $\mathcal{L}^2 $ is Lebesgue measure on $\mathbb{T}^2$. Thus, for any $x \in W^{1,\infty} $ there exists a sequence $\{x_n\} \in C^1$ such that
    $x_n \overset{\star} {\rightharpoonup} x  $. Since $C^1$ is separable, this shows that $W^{1,\infty} $ is separable under weak$\star$ topology.

    We denote by $\mathcal{T}_n,\mathcal{T}_{bw\star}, \mathcal{T}_{w\star}$ the norm topology, the bounded weak$\star$ topology, and the weak$\star$ topology of 
    $W^{1,\infty}$, respectively. By \cite{megginson2012introduction}, we have 
    \begin{equation}\label{subset:topology}
        \mathcal{T}_{w\star}\subsetneq \mathcal{T}_{bw\star}\subsetneq \mathcal{T}_n.    \end{equation}
   By \cite{megginson2012introduction} and the same argument as in \cite{bessaih2020invariant}, we have that $f:W^{1,\infty} \rightarrow \mathbb{R} $ is $\mathcal{T}_{bw\star} $-continuous if and only if it is sequentially $\mathcal{T}_{w\star} $-continuous. Denoting by $C(W^{1,\infty} ,\mathcal{T}), SC(W^{1,\infty} ,\mathcal{T}) $ the space of all functions $f: W^{1,\infty} \rightarrow \mathbb{R}$ that are $\mathcal{T}$-continuous and sequentially $\mathcal{T}$-continuous, respectively. Thus, we have that 
    $$C(W^{1,\infty} ,\mathcal{T}_{w\star})\subsetneq C(W^{1,\infty} ,\mathcal{T}_{bw\star} )= SC(W^{1,\infty} ,\mathcal{T}_{w\star})\subsetneq C(W^{1,\infty} ,\mathcal{T}_n). $$ 
We write $E=C(\mathbb{T}^2)$. Since the embedding
\[
W^{1,\infty}(\mathbb{T}^2)\hookrightarrow E
\]
is compact, every subsequence of $(x_n)$ admits a further subsequence converging strongly in $E$. As $x_n \overset{\ast}{\rightharpoonup} x$ in $W^{1,\infty}(\mathbb{T}^2)$, every such limit must coincide with $x$. Hence,
\[
x_n \to x \quad \text{strongly in } E.
\]
Therefore, the embedding
\[
i:(W^{1,\infty},T_{w^\star})\to C(\mathbb T^2)
\]
is sequentially continuous. Since
\[
C(W^{1,\infty},T_{bw^\star})
=
SC(W^{1,\infty},T_{w^\star}),
\]
it follows that the embedding
\[
i:(W^{1,\infty},T_{bw^\star})\to C(\mathbb T^2)
\]
is continuous. Denoting by $C(E)$ the space of all functions $f: E \rightarrow \mathbb{R}$ that are continuous under the norm topology of $E$. 
    We also have 
    \begin{equation}\label{C(E)}
        C(E)\subset SC(W^{1,\infty} ,\mathcal{T}_{w\star})=C(W^{1,\infty} ,\mathcal{T}_{bw\star} ) .    \end{equation}

Let us denote by $\mathcal{B}(\mathcal{T})$ the $\sigma$-algebra of Borelian subsets of $W^{1,\infty}$ with respect to a given topology $\mathcal{T}$. According to (\ref{subset:topology}) and the same arguments as in Lemma 1 of \cite{bessaih2020invariant}, we have $$\mathcal{B}(\mathcal{T}_{w\star})=\mathcal{B}(\mathcal{T}_{bw\star} ) \subset  \mathcal{B}(\mathcal{T}_n ).$$

\section{Global-in-time boundedness for solution}
In this section, we will prove a bound for the $L^\infty(0,\infty;W^{1,\infty})$ norm of the solution $w$ when the damping term $\gamma$ is large. We need to mention that Bessaih and Ferrario \cite{bessaih2020invariant} showed a bound for the $L^\infty(0,T;W^{1,4})$ norm of the solution for the stochastic Euler equation with Gaussian noise. However, the bound they proved does not hold uniformly for $T\geq 0$. The global-in-time $W^{1,\infty}$ estimate provides a compactness mechanism for
the dynamics: for each initial condition $w_0\in W^{1,\infty}$, the corresponding
trajectory remains in a bounded subset of $W^{1,\infty}$, which is relatively
compact in $C(\mathbb T^2)$. This compactness structure provides the compactness ingredient required for the discrete-time mixing argument developed below. To establish this estimate, we first prove several auxiliary lemmas.
\begin{lemma}\label{Lemma:K Lp}
  Let $p_0>2$ be fixed. If $w\in L^p(\mathbb T^2)$ with $p\ge p_0$ and $u= \mathcal{K}\ast w$, then
\[
\|\mathcal{K}*w\|_\infty \le C_{p_0}\|w\|_p .
\]
\end{lemma}
\begin{proof}
    Let $v$ be the solution for the equation $\nabla^\perp v = u.$ Then, $v$ satisfies the elliptic system
    $$\Delta v =w.$$
    Hence, by Schauder estimates and the Sobolev inequality, for some $2<p^\prime<p$
    $$\|u\Vert_\infty  \leq C\|\nabla u\Vert_{p_0} = C\|\nabla^2 v\Vert_{p_0} \leq C_{p_0}\|\Delta v\Vert_{p_0}\leq C_{p_0}\|w\Vert_p.$$
\end{proof}

We introduce the Ornstein-Uhlenbeck process
    \begin{equation}\label{eqn:OU}
        \partial_tZ_\lambda(t)= -\lambda Z_{\lambda}(t)+\eta(t),\; Z_\lambda(0)=0,
    \end{equation}
for $\lambda>0$. Then, we have $$Z_\lambda(t)= \int_{0}^{t}e^{-\lambda(t-s)}\eta(s)\;ds.$$ 
Indeed, using the uniform bound on each unit interval and the exponential decay, we have $\mathbb{P}-a.s. $
\begin{equation}\label{inequality:Z lambda}
    \sup_{t \geq 0}\|Z_\lambda(t)\Vert_{H^a} \leq C \sup_{t \geq 0}\int_{0}^{t}e^{-\lambda(t-s)}\|\eta(s)\Vert_{H^a}\;ds \leq B\sum_{j=0}^{\infty} e^{-\lambda j}
\le \frac{C B}{\lambda},
\end{equation} for some $a>3$.
\begin{lemma}\label{lemma:w infity}
  Let $\gamma>0$. 
  If $w_0 \in W^{1,\infty}$, then 
 \begin{equation*}
     \sup_{t\geq0}\| w(t)\Vert_{\infty} < \infty,\; \mathbb{P}-a.s. \end{equation*}
     Moreover, there exist $R(\gamma)$ such that 
     \begin{equation}\label{inequality:R}
          \|w(t)\Vert_{L^\infty} \leq  \|w_0\Vert_\infty\cdot e^{-\frac{\gamma t}{2}}+R,\;  \mathbb{P}-a.s. .     
        \end{equation}
    
\end{lemma}
\begin{proof}
 Defining $Y_\lambda=w-Z_\lambda$, we get 
\begin{equation}\label{eqn:Y}
    \partial_tY_\lambda+\gamma Y_\lambda(t)=-B(u(t),w(t)) +(\lambda-\gamma) Z_\lambda(t).
\end{equation}
Multiplying both sides of equation (\ref{eqn:Y}) with $ Y_\lambda(t) \cdot | Y_\lambda(t)\vert^{p-2} $ , $p>2$ and then integrating over $\mathbb{T}^2$; we get
\begin{align}
    \frac{1}{p}\partial_t\| Y_\lambda\Vert_{p}^{p}+ \gamma \| Y_\lambda(t)\Vert_{p}^{p}&= -\langle u(t)\cdot \nabla Y_\lambda(t)+ u(t)\cdot \nabla Z_\lambda(t), Y_\lambda(t) \cdot | Y_\lambda(t)\vert^{p-2}\rangle \nonumber \\&\ \  + \langle  (\lambda-\gamma)  Z_\lambda(t), Y_\lambda(t) \cdot | Y_\lambda(t)\vert^{p-2}  \rangle.
    \end{align}
    Since $\nabla \cdot  u(t)=0$, we have $$ \langle  u(t)\cdot \nabla Y_\lambda(t), Y_\lambda(t) \cdot | Y_\lambda(t)\vert^{p-2} \rangle = \frac1p \langle u(t),\nabla |Y_\lambda(t)\vert^p \rangle= -\frac1p\langle \nabla \cdot u(t),|Y_\lambda(t)\vert^p \rangle=0.$$
Then, for some $C>1$
\begin{align*}
    \frac{1}{p}\partial_t\| Y_\lambda\Vert_{p}^{p} &+\gamma \| Y_\lambda(t)\Vert_{p}^{p}  = \langle  (\lambda-\gamma)  Z_\lambda(t)- u(t)\cdot \nabla Z_\lambda(t), Y_\lambda(t) \cdot | Y_\lambda(t)\vert^{p-2}\rangle \\ &\leq  \|u(t)\Vert_p\|\nabla Z_\lambda(t)\Vert_\infty\|Y_\lambda(t)\Vert_p^{p-1} + |\lambda-\gamma\vert\cdot\|Z_\lambda(t)\Vert_p\cdot  \|Y_\lambda(t)\Vert_p^{p-1}\\
    &\leq \big[C(\|Y_\lambda(t)\Vert_p+\|Z_\lambda(t)\Vert_p)\|\nabla Z_\lambda(t)\Vert_\infty    + |\lambda-\gamma\vert\cdot\|Z_\lambda(t)\Vert_p \big] \|Y_\lambda(t)\Vert_p^{p-1}   \end{align*}
Here, we used Lemma \ref{Lemma:K Lp} in the last inequality. Since $\partial_t\| Y_\lambda\Vert_{p}^p= p \|Y_\lambda\Vert_{p}^{p-1}\cdot \partial_t\| Y_\lambda\Vert_{p}  $, we deduce that 
$$\partial_t\| Y_\lambda\Vert_{p} + \gamma \| Y_\lambda(t)\Vert_{p} \leq  C(\|Y_\lambda(t)\Vert_p+\|Z_\lambda(t)\Vert_p)\|\nabla Z_\lambda(t)\Vert_\infty   + |\lambda-\gamma\vert\cdot\|Z_\lambda(t)\Vert_p. $$ Thus, by the Sobolev inequality
$$\partial_t\| Y_\lambda\Vert_{p}\leq (C\|Z_\lambda(t)\Vert_{H^a}-\gamma )\cdot \| Y_\lambda(t)\Vert_{p}+ C( \|Z_\lambda(t)\Vert_{H^a}  + |\lambda-\gamma\vert) \cdot\|Z_\lambda(t)\Vert_{H^a} .$$ 
Applying Gr{\"o}nwall inequality, we have for $t>0$ and arbitrary $p>2$
\begin{align*}
\| Y_\lambda(t)\Vert_{p} &\leq \|w(0)\Vert_p\cdot e^{-\int_0^t (\gamma -C\|Z_\lambda(s)\Vert_{H^a}\;) ds } \\ &\ \ + \int_0^t C( \|Z_\lambda(s)\Vert_{H^a}  + |\lambda-\gamma\vert) \cdot\|Z_\lambda(s)\Vert_{H^a} \cdot  e^{-\int_s^t( \gamma -C\|Z_\lambda(r)\Vert_{H^a})\; dr } \;ds.
\end{align*}
Now, taking $p \rightarrow \infty$, the above inequality implies that  
\begin{align*}
\| Y_\lambda(t)\Vert_{\infty} &\leq \|w(0)\Vert_\infty\cdot e^{-\int_0^t (\gamma -C\|Z_\lambda(s)\Vert_{H^a}\;) ds } \\ &\ \ + \int_0^t C( \|Z_\lambda(s)\Vert_{H^a}  + |\lambda-\gamma\vert) \cdot\|Z_\lambda(s)\Vert_{H^a} \cdot  e^{-\int_s^t (\gamma -C\|Z_\lambda(r)\Vert_{H^a}\;) dr } \;ds.
\end{align*} 
Choosing $\lambda$ sufficiently large such that $\frac{CB}{\lambda}\leq\frac \gamma 2 \wedge \frac{CB}{\gamma} $. By (\ref{inequality:Z lambda}), we have $\mathbb{P}-a.s. $
\begin{equation}\label{inequality: w infty}
\begin{split}
   &\ \ \ \  \|w(t)\Vert_{\infty}\\&\leq  \| Y_\lambda(t)\Vert_{\infty} + \|Z_\lambda(t)\Vert_{\infty} \\ &\leq \|w(0)\Vert_\infty\cdot e^{-(\gamma-\frac{CB}{\lambda})t}  + \frac{\frac{C^2B}{\lambda} }{\gamma-\frac{CB}{\lambda}}\cdot (\frac{CB}{\lambda}+  \frac{4C^2B^2} {\gamma^2}  ) + \frac{CB}{\lambda} \\ &\leq \|w(0)\Vert_\infty\cdot e^{-\frac{\gamma t}{2}}  + \frac{\frac{CB}{\gamma} }{\gamma/2} (\frac{CB}{\gamma}+  \frac{2C^2B^2} {\gamma^2}  ) + \frac{CB}{\gamma} .    \end{split}
    \end{equation}
Then, we get the desired result by setting 
\begin{equation}\label{eqn:R}
    R(\gamma)= \frac{\frac{CB}{\gamma} }{\gamma/2} (\frac{CB}{\gamma}+  \frac{2C^2B^2} {\gamma^2}  ) + \frac{CB}{\gamma} .
    \end{equation}

\end{proof}
 We can now prove the main result of this section.

\begin{theorem}\label{Theorem:w bounded}
Let $\gamma\geq0$ and $2<p\leq\infty$. If $w_0\in W^{1,p}$, then for any
$T>0$,
\[
w\in L^\infty(0,T;W^{1,p})\cap C_{w^\star}([0,T];W^{1,p}),
\qquad \mathbb P\text{-a.s.}
\]
Here and throughout this theorem, for $2<p<\infty$,
$C_{w^\star}([0,T];W^{1,p})$ is understood as
$C_w([0,T];W^{1,p})$.

Moreover, for every $2<p\le\infty$, if $w_0\in W^{1,p}$, there exists
$\gamma^\star(w_0)>0$ such that, for every $\gamma\ge\gamma^\star$,
\[
\sup_{t\ge0}\|\nabla w(t)\|_{p}<\infty,
\qquad \mathbb P\text{-a.s.}
\]

\end{theorem}
\begin{proof}
We start from equation (\ref{eqn:Y}). Throughout this proof, we fix \(\lambda\ge1\) satisfying
\[
|\lambda-\gamma|\le1.
\]
Taking the gradient over both sides of the equation (\ref{eqn:Y}), multiplying by $\nabla Y_{\lambda}(t) \cdot |\nabla Y_{\lambda}(t)\vert^{p-2} $, $p>2$ and then integrating over $\mathbb{T}^2$; we get
\begin{align}\label{eqn:gradient Y}
    \frac{1}{p}\partial_t\|\nabla Y_{\lambda}\Vert_{p}^{p}+ \gamma \|\nabla Y_{\lambda}(t)\Vert_{p}^{p}&= -\langle \nabla[u(t)\cdot \nabla Y_{\lambda}(t)]+ \nabla[u(t)\cdot \nabla Z_{\lambda}(t)], \nabla Y_{\lambda}(t) \cdot |\nabla Y_{\lambda}(t)\vert^{p-2}\rangle \nonumber\\&\ \  + \langle (\lambda-\gamma) \nabla Z_{\lambda}(t),\nabla Y_{\lambda}(t) \cdot |\nabla Y_{\lambda}(t)\vert^{p-2}  \rangle.
 \end{align}
Note that 
\begin{align*}
    \langle \nabla[u(t)\cdot \nabla Y_{\lambda}(t)],\nabla Y_{\lambda}(t) \cdot |\nabla Y_{\lambda}(t)\vert^{p-2}\rangle &=  \langle \nabla u(t)\cdot \nabla Y_{\lambda}(t),\nabla Y_{\lambda}(t) \cdot |\nabla Y_{\lambda}(t)\vert^{p-2}\rangle \\&\ \ +\sum_{j=1}^2 \sum_{i=1}^2 \langle u_j(t)\cdot\partial_i\partial_jY_{\lambda}(t),\partial_iY_{\lambda}(t) \cdot |\nabla Y_{\lambda}(t)\vert^{p-2} \rangle\\&:=I_1(t)+I_2(t).
\end{align*}
By integration by parts, we have
\begin{align*}
    &\ \ \ \ I_2(t)\\&= -\sum_{i=1}^2\langle\sum_{j=1}^2\partial_ju_j(t)\cdot \partial_i Y_{\lambda}(t),\partial_iY_{\lambda}(t) \cdot |\nabla Y_{\lambda}(t)\vert^{p-2} \rangle \\&\ \ - \sum_{i,j}\langle u_j(t)\cdot \partial_iY_{\lambda}(t), \partial_j[\partial_iY_{\lambda}(t) \cdot |\nabla Y_{\lambda}(t)\vert^{p-2} ] \rangle\\&=0-\sum_{i,j}\langle u_j(t)\cdot \partial_iY_{\lambda}(t), \partial_j\partial_iY_{\lambda}(t) \cdot |\nabla Y_{\lambda}(t)\vert^{p-2}  \rangle\\&\ \ - \sum_{i,j}\langle u_j(t)\cdot \partial_iY_{\lambda}(t), \partial_iY_{\lambda}(t) \cdot \partial_j[|\nabla Y_{\lambda}(t)\vert^{p-2}]  \rangle \\&= -I_2(t)-\sum_{j=1}^2 \langle u_j(t)\cdot |\nabla Y_{\lambda}(t)\vert^2,   \partial_j[|\nabla Y_{\lambda}(t)\vert^{p-2}]  \rangle = -(p-1)I_2(t).  \end{align*}
    Thus, $I_2(t)\equiv0$.
    Moreover, $$|I_1(t)\vert=| \langle \nabla u(t)\cdot \nabla Y_{\lambda}(t),\nabla Y_{\lambda}(t) \cdot |\nabla Y_{\lambda}(t)\vert^{p-2}\rangle\vert \leq C\|\nabla u(t)\Vert_{L^\infty}\|\nabla Y_{\lambda}(t)\Vert^{p}_{p}. $$
    Next, we need an estimate for $\|\nabla u\Vert_p$. We fix $2<p^\prime<p $. Then, by Lemma 3.1 of \cite{kato1986} and the Sobolev inequality, one obtains that for some $1+\frac{2}{p^\prime}<s<2 $,
    $$\|\nabla u(t)\Vert_p \leq C_{p^\prime}\|\nabla u(t)\Vert_{W^{s-1,p^\prime}} \leq C_{p^\prime}\|w(t)\Vert_{W^{s-1,p^\prime}}\leq C_{p^\prime}\|\nabla w(t)\Vert_{p^\prime} \leq C \|\nabla w(t)\Vert_{p}. $$
    By H{\"o}lder inequality and Young inequality, we have
    \begin{align*}
        &\ \ \ \ |\langle\nabla[u(t)\cdot \nabla Z_{\lambda}(t)], \nabla Y_{\lambda}(t) \cdot |\nabla Y_{\lambda}(t)\vert^{p-2}\rangle \vert\\ &\leq C\|Z_{\lambda}(t)\Vert_{W^{2,\infty} }\|\nabla u(t)\Vert_{p}\|\nabla Y_{\lambda}(t)\Vert_{p}^{p-1} \\ &\leq C\|Z_{\lambda}(t)\Vert_{W^{2,\infty} }\|\nabla  w(t)\Vert_{p}\|\nabla Y_{\lambda}(t)\Vert_{p}^{p-1} \\ & \leq C\|Z_{\lambda}(t)\Vert_{W^{2,\infty} }(\|\nabla Y_{\lambda}(t)\Vert_{p}+\|\nabla Z_{\lambda}(t)\Vert_{p} )\|\nabla Y_{\lambda}(t)\Vert_{p}^{p-1}.
        \end{align*} 
        Similarly, we have \begin{align*}| \langle (\lambda-\gamma) \nabla Z_{\lambda},\nabla Y_{\lambda}(t) \cdot |\nabla Y_{\lambda}(t)\vert^{p-2}  \rangle \vert &\leq |\lambda-\gamma\vert\cdot\|Z_{\lambda}(t)\Vert_{W^{1,p}}\cdot \|\nabla Y_{\lambda}(t)\Vert_{p}^{p-1}. 
        \end{align*}
Again, by Lemma 3.1 of \cite{kato1986}, we have the estimate for $\|\nabla u(t)\Vert_{\infty}$,
\begin{equation}\label{estimate:nable u}
    \|\nabla u(t)\Vert_{L^\infty}\leq C_p\|w(t)\Vert_{L^\infty} \left[1+ \log\left(1+\frac{\|\nabla w(t)\Vert_{p}}{\|w(t)\Vert_{L^\infty}}\right) \right].
\end{equation}
Since $\partial_t\|\nabla Y_{\lambda}\Vert_{p}^{p}= p \|\nabla Y_{\lambda}\Vert_{p}^{p-1}\cdot \partial_t\|\nabla Y_{\lambda}\Vert_{p}  $, by (\ref{eqn:gradient Y}), the Sobolev inequality and the above estimates, we obtain for $ t \in [0,T]$ and $a>3$,
\begin{align*}
   &\ \ \ \  \partial_t\|\nabla Y_{\lambda}\Vert_{p}+ \gamma \|\nabla Y_{\lambda}(t)\Vert_{p}\\ &\leq C_p\left(\|w(t)\Vert_{L^\infty} \left[1+ \log\left(1+\frac{\|\nabla w(t)\Vert_{p}}{\|w(t)\Vert_{L^\infty}}\right) \right]+ \|Z_{\lambda}(t)\Vert_{H^a}  \right)\cdot \|\nabla Y_{\lambda}(t)\Vert_{p} \\&\ \ + C (\|Z_{\lambda}(t)\Vert_{H^a }^{2}
   + |\lambda-\gamma\vert\cdot \|Z_{\lambda}(t)\Vert_{H^a } ). 
    \end{align*}

Since 
$\log(x+y)\leq \log_+(x+y) \leq \log2+\log_+x+\log_+y$ for $x,y>0 $ and $C\geq 1$, we get
\begin{equation}
\begin{split}\label{inequality: Y differential}
   &\ \ \ \  \partial_t\|\nabla Y_{\lambda}\Vert_{p}\\ &\leq \left(-\gamma+C_p\|w(t)\Vert_{L^\infty} \left[1+ \log\left(1+\frac{\|\nabla w(t)\Vert_{p}}{\|w(t)\Vert_{L^\infty}}\right) \right]+ C_p\|Z_{\lambda}(t)\Vert_{H^a}  \right)\cdot \|\nabla Y_{\lambda}(t)\Vert_{p} \\&\ \ + C \|Z_{\lambda}(t)\Vert_{H^a }^{2} \\ &\leq \Bigg[ -\gamma + C_p\|w(t)\Vert_{L^\infty} \left[1+ \log_+\left(\|\nabla Y_{\lambda}(t)\Vert_{p}\right) + \log_+(\|\nabla Z_{\lambda}(t)\Vert_{p} ) + \log_+(\frac{1}{\|w(t)\Vert_{L^\infty} }) \right] \\&\ \  +C_p \|Z_{\lambda}(t)\Vert_{H^a }\Bigg]\cdot \|\nabla Y_{\lambda}(t)\Vert_{p} + C (\|Z_{\lambda}(t)\Vert_{H^a }^{2}
   + |\lambda-\gamma\vert\cdot \|Z_{\lambda}(t)\Vert_{H^a } ) .
   \end{split}
    \end{equation}
    
Setting $X_p(t)= \log(\|\nabla Y_{\lambda}(t)\Vert_{p} \vee 1 )= \log_+ (\|\nabla Y_{\lambda}(t)\Vert_{p}) $, we have
\begin{equation}\label{eqn:X}
    \partial_tX_p = \frac{1}{\|\nabla Y_{\lambda}(t)\Vert_{p} }\chi_{\{\|\nabla Y_{\lambda}(t)\Vert_{p}\geq 1\} } \partial_t \|\nabla Y_{\lambda}\Vert_{p} .\end{equation}
By substituting (\ref{eqn:X}) into (\ref{inequality: Y differential}) and the Sobolev embedding, we obtain \begin{align*}
    \partial_t X_p &\leq \Bigg[ -\gamma + C_p\|w(t)\Vert_{L^\infty} \left[1 + \log_+(\|Z_{\lambda}(t)\Vert_{H^a} ) + \log_+(\frac{1}{\|w(t)\Vert_{L^\infty} }) \right] \\&\ \  +C_p \|Z_{\lambda}(t)\Vert_{H^a } + C_p\|w(t)\Vert_{L^\infty}\cdot X_p(t) +C \|Z_{\lambda}(t)\Vert_{H^a }^{2} \Bigg]\cdot \chi_{\{\|\nabla Y_{\lambda}(t)\Vert_{p}\geq 1\}} . \end{align*}
Since \(X_p(t)=\log(\|\nabla Y_{\lambda}(t)\|_p\vee1)\) is identically zero on the set where \(\|\nabla Y_{\lambda}(t)\|_p\le1\), it suffices to estimate \(X_p\) on the set where \(\|\nabla Y_{\lambda}(t)\|_p\geq1\). All the following differential inequalities are understood to hold for a.e. \(t\) on this set. 

We take $p=p^\star$ for some $p^\star>2$ fixed once and for all, by Gr{\"o}nwall inequality,
\begin{equation}\label{inequality:Gronwall Y}
    X_{p^\star}(t) \leq  \Bigg[ \alpha(t)+C_{p^\star}\int_0^t \alpha(s)  \|w(s)\Vert_{\infty}e^{C_{p^\star} \int_s^t\|w(r)\Vert_{\infty}\;dr } \;ds\Bigg]\vee 0 ,   
    \end{equation}
    where 
    \begin{equation}\label{def:alpha}
    \begin{split}
 \alpha(t)&= \log(\|\nabla w_0\Vert_{p^\star}\vee 1 )+ \int_0^t \Bigg[-\gamma+ C_{p^\star}\big[\|w(s)\Vert_{L^\infty}  +\|w(s)\Vert_{L^\infty}\cdot \log_+(\| Z_{\lambda}(s)\Vert_{H^a})  \\&\ \  + \|Z_{\lambda}(s)\Vert_{H^a }+ \|w(s)\Vert_{L^\infty}\cdot \log_+(\frac{1}{\|w(s)\Vert_{L^\infty} }) \big]+C (\|Z_{\lambda}(t)\Vert_{H^a }^{2}
   + |\lambda-\gamma\vert\cdot \|Z_{\lambda}(t)\Vert_{H^a } )\Bigg] \;ds.
    \end{split}    
        \end{equation}
Since $-x\log x \leq \frac{1}{e}$ for $x,y>0 $, it follows from (\ref{inequality:Gronwall Y}), (\ref{def:alpha}), Theorem \ref{existence and uniqueness} and Lemma \ref{lemma:w infity} that $w \in L^\infty(0,T;W^{1,p^\star})$. Then, we can obtain the uniform estimate for $\|\nabla Y_{\lambda}\Vert_{p}$ with respect to all $p>2$ by substituting $\|\nabla w(t)\Vert_{p^\star}$ into (\ref{estimate:nable u}). By the same computations as above, since $x\log(1+\frac{y}{x})\leq y$ for $x,y >0$, we have that for fixed $p^\star>2$ and any $p \geq p^\star$,
    \begin{equation}\label{inequality:Y W1,p}  
    \begin{split}
    \partial_t\|\nabla Y_{\lambda}(t)\Vert_{p} &\leq \Big(-\gamma+ C_{p^\star}[\|w(s)\Vert_{L^\infty} + \|\nabla w(s)\Vert_{p^\star} \\&\ \ + \|Z_{\lambda}(s)\Vert_{H^a }]  \Big)\cdot \|\nabla Y_{\lambda}(t)\Vert_{p} + C (\|Z_{\lambda}(t)\Vert_{H^a }^{2}
   + |\lambda-\gamma\vert\cdot \|Z_{\lambda}(t)\Vert_{H^a } ).\end{split}   \end{equation}
    Taking $p \rightarrow \infty$, 
    
    \begin{equation}\label{inequality:Y W1,infty}
    \begin{split}
         \partial_t\|\nabla Y_{\lambda}(t)\Vert_{\infty} &\leq \Big(-\gamma+ C_{p^\star}[\|w(s)\Vert_{L^\infty} + \|\nabla w(s)\Vert_{p^\star} \\&\ \ + \|Z_{\lambda}(s)\Vert_{H^a }]  \Big)\cdot \|\nabla Y_{\lambda}(t)\Vert_{\infty} +C (\|Z_{\lambda}(t)\Vert_{H^a }^{2}
   + |\lambda-\gamma\vert\cdot \|Z_{\lambda}(t)\Vert_{H^a } ).     \end{split}
        \end{equation}
    Hence, we get $w \in L^\infty(0,T;W^{1,\infty})$. Going back to equation (\ref{eqn:Y}) and using the regularity of $w$, we get that $\mathbb{P}-a.s.$, $w \in H^{1,2}(0,T;L^2)$. By Lemma 1.4 of Chapter 3 in \cite{temam2024navier}, we also have $w \in C_{w^\star}([0,T];W^{1,\infty}),\; \mathbb{P}-a.s. $ 

    Next, we will show that for any $w_0 \in W^{1, p}$ there exists $\gamma*>0$ such that for any $\gamma \geq \gamma^\star$, $2<p \leq \infty$,
$$\sup_{t\geq0}\|\nabla w(t)\Vert_{p}< \infty,\; \mathbb{P}-a.s. $$ 
We first prove the estimate for a fixed finite exponent $p^\star>2$.
This bound will then be used to control the logarithmic term in the
$W^{1,\infty}$ estimate.
According to Lemma \ref{lemma:w infity}, for any initial value $w_0 \in W^{1,p^\star}\subset L^\infty $, there exist $R(\gamma)$ that are decreasing with respect to $\gamma$ such that (\ref{inequality:R}) holds  $\mathbb{P}-a.s.$.  
Recall that for any $\gamma >0$, we set $\gamma\geq 1$ with $|\lambda-\gamma|\leq1 $. 
Therefore, for any $ \gamma >0$, 
$$ \sup_{t \geq 0}\|Z_{\lambda}(t)\Vert_{H^a } \leq \frac{CB}{|\gamma-1|} \wedge CB:= \bar{R}(\gamma) .$$
Thus, there exists $M_1>0$ depending on $\gamma$ such that for all $t\geq 0$, $\mathbb{P}-a.s. $
 \begin{equation}\label{inequality:C1}
 \begin{split}
    F(t) = C_{p^\star}( \|Z_{\lambda}(t)\Vert_{H^a }&+ \|w(t)\Vert_{L^\infty}\cdot \log_+(\frac{1}{\|w(t)\Vert_{L^\infty} })  )\\&+ C (\|Z_{\lambda}(t)\Vert_{H^a }^{2}
   + |\lambda-\gamma\vert\cdot \|Z_{\lambda}(t)\Vert_{H^a } ) \leq M_1(\gamma). 
   \end{split}
\end{equation}
Then, by (\ref{def:alpha}), we have $\mathbb{P}-a.s.$
\begin{equation}\label{inequality:alpha 1}
\begin{split}
    \alpha(t) &\leq \log_+(\|\nabla w_0\Vert_{p^\star}) -(\gamma-M_1)t + C_{p^\star}(1+\log_+\bar{R}) \int_0^t \|w(s)\Vert_{L^\infty}\;ds\\&
    \leq \log_+(\|\nabla w_0\Vert_{p^\star})- (\gamma-M_1)t + C_{p^\star}(1+\log_+\bar{R}) \int_0^t \|w_0\Vert_{L^\infty}\cdot e^{-\frac{\gamma s}{2}} +R\;ds\\&
    \leq \log_+(\|\nabla w_0\Vert_{p^\star})+ \frac{2 C_{p^\star}(1+\log_+\bar{R}) \|w_0\Vert_{L^\infty}}{\gamma} - (\gamma-C_{p^\star}R(\log_+\bar{R} +1) -M_1)t 
    \end{split}
\end{equation}
Recall that $R(\gamma), \bar{R}(\gamma)$ and $ M_1(\gamma)$ are decreasing in $\gamma$ from (\ref{eqn:R}). Thus,
we are able to set $ \gamma > C_{p^\star}R(\log_+\bar{R}+1) +M_1 $. Therefore,
\begin{equation}\label{inequality:alpha 2}
   \lim_{t \rightarrow \infty}\alpha(t)=-\infty,\; \sup_{t \geq0}\alpha(t) < \log_+(\|\nabla w_0\Vert_{p^\star})+ \frac{2 C_{p^\star}(1+\log_+\bar{R}) \|w_0\Vert_{L^\infty}}{\gamma} .
   \end{equation} 
Hence, for any $M_2> 0$, there exists $$t_1 = \frac{\log_+(\|\nabla w_0\Vert_{p^\star})+ \frac{2 C_{p^\star}(1+\log_+\bar{R}) \|w_0\Vert_{L^\infty}}{\gamma}+M_2}{\gamma-C_{p^\star}R(\log_+\bar{R}+1) -M_1 }. $$ such that $\alpha(t)  \leq -M_2$ for all $ t\geq t_1, \;\mathbb{P}-a.s.$.
Therefore, we get for any $t\geq t_1$,
\begin{align}\label{estimate:Y R4}
   &\ \ \ \  C_{p^\star}\int_0^t \alpha(s)  \nonumber \|w(s)\Vert_{L^\infty}e^{C_{p^\star} \int_s^t\|w(r)\Vert_{L^\infty}\;dr } \;ds\\ &= C_{p^\star}\int_0^{t_1} \alpha(s)  \|w(s)\Vert_{L^\infty}e^{C_{p^\star} \int_s^t\|w(r)\Vert_{L^\infty}\;dr } \;ds    + C_{p^\star}\int_{t_1}^t \alpha(s)  \|w(s)\Vert_{L^\infty}e^{C_{p^\star} \int_s^t\|w(r)\Vert_{L^\infty}\;dr } \;ds \nonumber \\ &\leq \sup_{t\geq 0}\alpha(t)\int_0^{t_1}  C_{p^\star}\|w(s)\Vert_{L^\infty}e^{C_{p^\star} \int_s^t\|w(r)\Vert_{L^\infty}\;dr } \;ds  - M_2\int_{t_1}^t   C_{p^\star}\|w(s)\Vert_{L^\infty}e^{C_{p^\star} \int_s^t\|w(r)\Vert_{L^\infty}\;dr } \;ds \nonumber \\ &\leq \sup_{t\geq 0}\alpha(t) \left(e^{C_{p^\star} \int_0^t\|w(s)\Vert_{L^\infty}\;ds }- e^{C_{p^\star} \int_{t_1}^t\|w(s)\Vert_{L^\infty}\;ds } \right)  -M_2\left( e^{C_{p^\star} \int_{t_1}^t\|w(s)\Vert_{L^\infty}\;ds }-1 \right) \nonumber \\& \leq \left[\sup_{t\geq 0}\alpha(t)(e^{C_{p^\star}\int_0^{t_1}(\| w_0\Vert_{\infty}\cdot e^{-\frac{\gamma s}{2}}+R)\;ds }-1)-M_2 \right]e^{C_{p^\star} \int_{t_1}^t\| w(s)\Vert_{\infty}\;ds } +M_2 \nonumber\\
   &\leq \Bigg[\left(\log_+(\|\nabla w_0\Vert_{p^\star})+ \frac{2 C_{p^\star}(1+\log_+\bar{R}) \|w_0\Vert_{L^\infty}}{\gamma}\right)\nonumber \\&\ \ \ \ \ \ \ \  \times (e^{\frac{2C_{p^\star}\| w_0\Vert_{\infty}}{\gamma}+ C_{p^\star} Rt_1 } -1)  -M_2 \Bigg]e^{C_{p^\star} \int_{t_1}^t\| w(s)\Vert_{\infty}\;ds } + M_2  . \end{align}
Here, we used (\ref{inequality:alpha 2}) in the last inequality and  (\ref{inequality:R}) in the third inequality.

For any $w_0 \in W^{1,p^\star}$, we choose $$M_2=C_1 \left(\log_+(\|\nabla w_0\Vert_{p^\star})+  \|w_0\Vert_{L^\infty}\right), $$ 
with $C_1$ sufficiently large. By the expression of $t_1$, for any bounded initial value $1\leq\|\nabla w_0\Vert_{p^\star} \leq C_0$, one can take $\gamma$ sufficiently large such that 
\begin{equation}\label{def:C1}
    (e^{\frac{2C_{p^\star}\| w_0\Vert_{\infty}}{\gamma}+ C_{p^\star} Rt_1 } -1) (1+ \frac{2 C_{p^\star}(1+\log_+\bar{R}) \|w_0\Vert_{L^\infty}}{\gamma} )\leq C_1.\end{equation}
In fact, we can set 
\begin{equation}\label{def:gamma}
    \gamma\geq \kappa(\log_+(\|\nabla w_0\Vert_{p^\star})\vee  \|w_0\Vert_{L^\infty} )+ C_\kappa+ C_{p^\star}R(\log_+\bar{R}+1) +M_1:=\bar{\gamma}, \end{equation}
where $\kappa>0$ is arbitrary. Here, $C_\kappa$ does not depend on $w_0$. Notice that the $R(\gamma)=O(\gamma^{-1} )$. Hence, 
$$C_{p^\star}t_1R \leq \frac{C_{p^\star}(1+C_1)}{\kappa}\cdot C_{\kappa}^{-1}, $$
and
\begin{align*}
   &\ \ \ \ (e^{\frac{2C_{p^\star}\| w_0\Vert_{\infty}}{\gamma}+ C_{p^\star} Rt_1 } -1) (1+ \frac{2 C_{p^\star}(1+\log_+\bar{R}) \|w_0\Vert_{L^\infty}}{\gamma} )\\ & \leq \left[ \exp{\left(\frac{2C_{p^\star}}{\kappa}+\frac{C_{p^\star}(1+C_1)}{\kappa C_\kappa}\right)-1 } \right](1+\frac{2C_{p^\star}}{\kappa})\\& \leq  \exp\left(\frac{2C_{p^\star}}{\kappa} \right) (1+\frac{2C_{p^\star}}{\kappa}) \exp \left(\frac{C_{p^\star}(1+C_1)}{\kappa C_\kappa}\right).
   \end{align*}
Taking $C_1= 2 \exp\left(\frac{2C_{p^\star}}{\kappa} \right) (1+\frac{2C_{p^\star}}{\kappa}) $ and $C_{\kappa}\geq \frac{C_{p^\star}(1+C_1)}{\kappa \ln 2} $, one obtains (\ref{def:C1}).
Then, (\ref{inequality:alpha 1}) and (\ref{estimate:Y R4}) yield that $\lim_{t\rightarrow\infty}\alpha(t)= -\infty$ and $$ \lim_{t\rightarrow\infty} \alpha(t)+C_{p^\star}\int_0^t \Big[\alpha(s)  \|w_\epsilon(s)\Vert_{\infty}e^{C_{p^\star} \int_s^t\|w_\epsilon(r)\Vert_{\infty}\;dr }\Big] \;ds=-\infty.  $$ 
By \eqref{inequality:Gronwall Y} and \eqref{inequality:C1}, there exists
\(t_2>0\) such that
\[
\|\nabla Y_\lambda(t_2)\|_{p^\star}\le 1,
\qquad
\sup_{t\in[0,t_2]}\|\nabla Y_\lambda(t)\|_{p^\star}<\infty .
\]
By Lemma \ref{lemma:w infity}, there exists \(t_3>0\), depending on
\(\|w_0\|_\infty\), such that
\[
\sup_{t\ge t_3}\|w(t)\|_\infty \le C R .
\]
Set \(T_0:=\max\{t_2,t_3\}\). If
\[
\|\nabla Y_\lambda(t)\|_{p^\star}<1
\qquad\text{for all }t\ge T_0,
\]
then the desired bound follows immediately. Otherwise, let
\[
\tau:=\inf\{t\ge T_0:\|\nabla Y_\lambda(t)\|_{p^\star}\ge1\}.
\]
By continuity, \(\|\nabla Y_\lambda(\tau)\|_{p^\star}=1\).

We now apply the differential inequality on every interval
\([a,b]\subset[\tau,\infty)\) on which
\[
\|\nabla Y_\lambda(t)\|_{p^\star}\ge1 .
\]
Since \(a\ge T_0\), we have \(\|w(t)\|_\infty\le CR\) for all \(t\in[a,b]\).
By the same argument as above, the coefficient \(\alpha(t)\) is non-positive on
\([a,b]\). Therefore
\[
\|\nabla Y_\lambda(t)\|_{p^\star}\le 1
\qquad\text{for all }t\in[a,b].
\]
Consequently,
\[
\sup_{t\ge T_0}\|\nabla Y_\lambda(t)\|_{p^\star}\le 1 .
\]
Using the uniform bound on \(Z_\lambda\), we obtain
\[
\sup_{t\ge T_0}\|\nabla w(t)\|_{p^\star}
\le 1+\bar R .
\]
Together with the bound on \([0,T_0]\), this gives
\[
\sup_{t\ge0}\|\nabla w(t)\|_{p^\star}<\infty .
\]    
By increasing $\bar\gamma$ if necessary, we may assume that for every
$\gamma\geq\bar\gamma$,
\[
\gamma>
C_{p^\star}(1+R(\gamma)+\bar R).
\]
We then set
\[
\gamma^\star:=\bar\gamma.
\]

Having obtained the global $W^{1,p^\star}$ bound, we now derive the
$W^{1,\infty}$ estimate. For any $\gamma>\gamma^\star$,
$$-\gamma+ C_{p^\star}[\|w(t)\Vert_{L^\infty} + \|\nabla w(t)\Vert_{p^\star} + \|Z_{\lambda}(t)\Vert_{H^a }] <-\sigma<0,\; t\geq t_2 ,$$ for some $\sigma>0$.
According to (\ref{inequality:Y W1,p}) and (\ref{inequality:Y W1,infty}), this implies $$ \sup_{t\geq0}\|\nabla w(t)\Vert_{p} < \infty,\; \mathbb{P}-a.s. $$ for $2<p\leq\infty$. It is clear from the construction that $\gamma^\star=\gamma^\star(w_0)$.

\end{proof}

\begin{remark}
From (\ref{def:gamma}), for every $\kappa>0$ there exists a constant
$C_\kappa$, independent of $w_0$, such that
\[
\gamma^\star
\le
\kappa\left(
\log_+\|\nabla w_0\|_{p^\star}
\vee
\|w_0\|_{L^\infty}
\right)
+C_\kappa.
\]
Consequently,
\[
\gamma^\star
=
o\!\left(
\log_+\|\nabla w_0\|_{p^\star}
\vee
\|w_0\|_{L^\infty}
\right)
\]
as
\[
\log_+\|\nabla w_0\|_{p^\star}
\vee
\|w_0\|_{L^\infty}\to\infty.
\]
\end{remark}
 The global-in-time bound for $W^{1,p}, 2<p<\infty$ norm also holds for any $\gamma>0$ if the noise intensity is small.
We denote by $w_\epsilon$ the solution to the following equation
\begin{equation}\label{eqn:Euler 2}
\begin{cases}
     d w_\epsilon(t)+\gamma w_\epsilon dt=(\mathcal{K}\ast w_\epsilon)\cdot \nabla w_\epsilon dt+ \epsilon d \eta(t),\\
     u_\epsilon= \mathcal{K}\ast w_\epsilon,\; \nabla\cdot u_\epsilon=0,\\     w_\epsilon(0)=w_0.
    \end{cases}
\end{equation}
\begin{corollary}\label{corollary:boundedness}
    Let $\gamma>0$ and $2<p \leq \infty$, there exists $\epsilon^\star(\gamma,w_0)>0$ such that for any $0\leq \epsilon < \epsilon^\star $
 \begin{equation}
     \sup_{t\geq0}\|\nabla w_\epsilon(t)\Vert_{p} < \infty,\; \mathbb{P}-a.s.. \end{equation}
\end{corollary}

\begin{proof}
   Following the same argument as in Lemma \ref{lemma:w infity} with $\frac{CB}{\lambda}\leq\frac \gamma 2 \wedge \frac{CB}{\gamma} $. Note that $\sup_{t \geq 0}\|Z_{\lambda,\epsilon}(t)\Vert_{H^a}^2  \leq \frac{\epsilon^2 C^2B^2}{\gamma} $. Then, by (\ref{inequality: w infty}), for $0\leq \epsilon \leq 1$,
   \begin{align*}
        \|w_\epsilon(t)\Vert_{\infty} &\leq \|w(0)\Vert_\infty\cdot e^{-(\gamma-\frac{\epsilon CB}{\lambda})t}  + \frac{\frac{\epsilon C^2B}{\lambda} }{\gamma-\frac{\epsilon CB}{\lambda}}\cdot (\frac{\epsilon CB}{\lambda}+  \frac{2C^2B^2} {\gamma^2}  ) + \frac{\epsilon CB}{\lambda} \\ &\leq \|w(0)\Vert_\infty\cdot e^{-\frac{\gamma t}{2}}  + \frac{\frac{\epsilon C^2B}{\lambda}}{\gamma-\epsilon\gamma/2} (\frac{\epsilon CB}{\lambda}+ \frac{2C^2B^2} {\gamma^2}  ) + \frac{\epsilon CB}{\gamma}. 
        \end{align*}
    We set $$ R_\epsilon= 2\frac{\frac{\epsilon C^2B}{\lambda}}{\gamma-\epsilon\gamma/2} (\frac{\epsilon CB}{\lambda}+ \frac{2C^2B^2} {\gamma^2}  ) + \frac{\epsilon CB}{\gamma}. $$ Hence, there exist $t_0(\|w_0\Vert_\infty)\geq 0$ such that  $ \mathbb{P}-a.s.  $
    \begin{equation}\label{inequality:R epsilon}
     \|w_\epsilon\Vert_{L^\infty(t_0,\infty;L^\infty)} \leq R_\epsilon\;\; and \; \; \|w_\epsilon\Vert_{L^\infty(0,\infty;L^\infty)} \leq  \|w_0\Vert_\infty \cdot e^{-\frac{\gamma t}{2}} +R_\epsilon.    \end{equation}
   Then, $R_\epsilon$ can be arbitrarily small as long as $\epsilon$ is small.  Next, recall the proof of Theorem \ref{Theorem:w bounded} and set $\gamma\geq 1$, $|\lambda-\gamma\vert \leq 1 $. Since $$\log(x+y) \leq \log((x+y)\vee \delta ) \leq \log 2+ \log(x\vee \delta) \vee \log(y \vee \delta ),$$
   and $$\log(xy\vee \delta) \leq \log (x \vee \delta) + \log (y \vee \delta)+ \log\frac{1}{\delta} $$
   for $x,y,\delta >0$, by (\ref{inequality: Y differential}), we have 
   \begin{align*}
      \partial_t\|\nabla Y_{\lambda,\epsilon}\Vert_p &\leq \Bigg[ -\gamma + C_p\|w_\epsilon(t)\Vert_{L^\infty} \Big[1+ \log\left(\|\nabla Y_{\lambda,\epsilon}(t)\Vert_{p}\vee \delta\right) \vee \log(\|\nabla Z_{\lambda,\epsilon}(t)\Vert_{p} \vee \delta )\\ &\ \ + \log(\frac{1}{\|w_\epsilon(t)\Vert_{L^\infty} }\vee \delta) + \log\frac{1}{\delta} \Big]   +C_p \|Z_{\lambda,\epsilon}(t)\Vert_{H^a }\Bigg]\cdot \|\nabla Y_{\lambda,\epsilon}(t)\Vert_{p} + F_{\epsilon}(t),    \end{align*} where $F_{\epsilon}(t):= C (\|Z_{\lambda,\epsilon}(t)\Vert_{H^a }^{2}
   + |\lambda-\gamma\vert\cdot \|Z_{\lambda,\epsilon}(t)\Vert_{H^a } ) $.
   
Setting $X_{\delta, \epsilon,p}(t)= \log(\|\nabla Y_{\lambda,\epsilon}(t)\Vert_{p} \vee \delta ) $, we have
\begin{equation}
    \partial_tX_{\delta,\epsilon,p} = \frac{1}{\|\nabla Y_{\lambda,\epsilon}(t)\Vert_{p} }\chi_{\{\|\nabla Y_{\lambda,\epsilon}(t)\Vert_{p}\geq \delta\} } \partial_t \|\nabla Y_{\lambda,\epsilon}\Vert_{p} .\end{equation}
Hence, for the time that $\|\nabla Y_{\lambda,\epsilon}(t)\Vert_{p} \leq \|\nabla Z_{\lambda,\epsilon}(t)\Vert_{p} $,  \begin{align*}
    \partial_t X_{\delta,\epsilon,p^\star} &\leq \Bigg[ -\gamma + C_{p^\star}\|w_\epsilon(t)\Vert_{L^\infty} \left[1  + \log(\frac{1}{\|w_\epsilon(t)\Vert_{L^\infty}  }\vee \delta ) +\log\frac{1}{\delta} \right] \\&\ \  +C_{p^\star} \|Z_{\lambda,\epsilon}(t)\Vert_{H^a } +  C_{p^\star}\|w_\epsilon(t)\Vert_{L^\infty}\cdot X_{\delta,\epsilon,p^\star}(t)  +F_\epsilon(t)/\delta \Bigg]\cdot \chi_{\{\|\nabla Y_{\lambda,\epsilon}(t)\Vert_{p}\geq \delta\}} . \end{align*} 
Since $\lim_{x \rightarrow 0^+}x \log(\frac{1}{x}\vee \delta)=0$, we have that for $t$ sufficiently large and $\epsilon$ sufficiently small,
$$C_{p^\star}\left(R+ R\cdot \log \frac{1}{\delta}+ \frac{R^2+R}{\delta}+ \|w_\epsilon(t)\Vert_{L^\infty}\cdot \log(\frac{1}{\|w_\epsilon(t)\Vert_{L^\infty} }\vee \delta )\right)< \gamma , \; \mathbb{P}-a.s.. $$ 
By following the same strategy as the proof of Theorem \ref{Theorem:w bounded}, we have that for any $w_0$, there exists $\epsilon^\star(\gamma,w_0)$ sufficiently small such that for $\epsilon<\epsilon^\star$ there exists $t^\star$ and for any $t\geq t^\star$,
$$\|\nabla Y_{\lambda,\epsilon}(t)\Vert_{p^\star} \leq \delta \vee \|\nabla Z_{\lambda,\epsilon}(t)\Vert_{p^\star}, \; \mathbb{P}-a.s.. $$ 
Since $R_\epsilon$ can be arbitrarily small as $\epsilon \rightarrow 0$, for any $t\geq t^\star$,  $$\|\nabla Z_{\lambda,\epsilon}(t)\Vert_{p^\star} \leq R_\epsilon\leq\delta. $$ 
We choose $0<4C_{p^\star}\delta<\gamma$. Then, for $\epsilon$ sufficiently small,
$$-\gamma+ C_{p^\star}[\|w_\epsilon(t)\Vert_{L^\infty} + \|\nabla w_\epsilon(t)\Vert_{p^\star} + \|Z_{\lambda,\epsilon}(t)\Vert_{H^a }] <-\sigma<0,\; t\geq t^\star ,$$ for some $\sigma>0$.     
By a similar argument, we get the desired result.
     
   


\end{proof}

\section{Existence of invariant measures}
In this section, we will prove the existence of invariant measures for equation (\ref{eqn:Euler 1}) by using Krylov-Bogoliubov's argument. In order to apply this method, the Feller property for the semigroup associated to the equation is required. In fact, we have the following result about the continuous dependence of solutions with respect to the initial data.
\begin{theorem}\label{Theorem: initial value continuous}
Let $\gamma\geq 0$. 
Given a sequence $\{x_n\}_{n}\in W^{1,\infty} $ which converges weakly$\star$ in $W^{1,\infty}$ to $x \in W^{1,\infty}$, we have
that, $\mathbb{P}$-a.s., for every $t>0$ the sequence $\{w(t;x_n)\}_{n}$ converges weakly$\star$ in $W^{1,\infty}$ to $w(t;x)$.
\end{theorem}
\begin{proof}
    We fix $\omega$ on a $\mathbb{P}-a.s.$ subset of $\Omega$ through the proof. By assumption, we have $x_n \in W^{1,\infty}$ and $\langle x_n,h\rangle_{W^{1,\infty},W^{-1,1}}\rightarrow \langle x,h\rangle_{W^{1,\infty},W^{-1,1}} $ for any $h \in W^{-1,1}$. Thanks to Theorem \ref{Theorem:w bounded} and the boundedness of $\{x_n\}$ in $W^{1, \infty}$, $w_n(\cdot)=w(\cdot;x_n)$ is bounded in $L^\infty(0,T;W^{1,\infty})$ for any $T>0$. Then, there exists a subsequence, which is still denoted by $w_n$, converging weakly$\star$ to some $w$ in $L^\infty(0,T;W^{1,\infty})$. In particular, for any $h \in W^{-1,1} $, we have 
    \begin{equation}\label{eqn:limit w int}
        \int_0^T \langle w_n(t),h\rangle_{W^{1,\infty},W^{-1,1}}\;dt\rightarrow \int_0^T\langle w(t),h\rangle_{W^{1,\infty},W^{-1,1}}\;dt ,
        \end{equation}
   as $n \rightarrow \infty$. Since (\ref{eqn:limit w int}) holds for every $T\geq0$, then for any given $T>0$ and a.e. $t \in [0,T], $  $$\langle w_n(t),h\rangle_{W^{1,\infty},W^{-1,1}}\rightarrow \langle w(t),h\rangle_{W^{1,\infty},W^{-1,1}}. $$
    Moreover, by the same argument as in Theorem 4 of \cite{bessaih2020invariant}, we can also show that $u_n$ converges to $u$ strongly in $C([0,T];H)$. To finish the proof, we need to show that $w$ is $w(\cdot;x)$. 
    
    Let $h \in H^1$, by the above argument, for a.e. $t \in [0,T] $ 
\begin{align*}
    \langle w_n(t),h \rangle_{W^{1,\infty},W^{-1,1}} &+ \gamma \int_0^t\langle w_n(s),h \rangle_{W^{1,\infty},W^{-1,1}}\;ds=-\int_0^t\langle u_n(s)\cdot \nabla w_n(s),h \rangle\;ds \\& \ \ + \langle x_n,h \rangle_{W^{1,\infty},W^{-1,1}}  + \int_0^t\langle \eta(t),h \rangle_{W^{1,\infty},W^{-1,1}}\;ds.  
    \end{align*}
Since $u \in C([0,T];H) $, we have $\sum_{i=1}^2\partial_i[u_i(\cdot)h] \in L^1(0,T;W^{-1,1}) $. Then, by integration by parts, 
\begin{align*}
    &\ \ \ \  \int_0^t\langle u_n(s)\cdot \nabla w_n(s),h \rangle\;ds- \int_0^t\langle u(s)\cdot \nabla w(s),h \rangle\;ds \\&= \int_0^t\langle [u_n(s)-u(s)]\cdot \nabla w_n, h(s) \rangle\;ds + \int_0^t\langle  [\nabla w_n(s)- \nabla w(s)],u(s)h \rangle\;ds\\&= -\int_0^t\langle [u_n(s)-u(s)]\cdot  w_n, \nabla h(s) \rangle\;ds  + \int_0^t \sum_{i=1}^2\langle  \partial_i[ w_n(s)-  w(s)],u(s)h_i \rangle\;ds\\&= -\int_0^t\langle [u_n(s)-u(s)]\cdot  w_n, \nabla h(s) \rangle\;ds -  \int_0^t\langle  w_n(s)-w(s),\sum_{i=1}^2\partial_i[u_i(s)h] \rangle_{W^{1,\infty},W^{-1,1}}\;ds \rightarrow 0  \end{align*}
  Here, we used the strong convergence of $u \in C([0,T];H) $, the boundedness of $\sup_{0\leq s \leq  T}\| w_n(s)\|_{\infty}$ and the weak$\star$ convergence of $ w_n \in L^\infty(0,T;W^{1,\infty}) $. As $n \rightarrow \infty$, we get for any $ h \in H^1$ and a.e. $t\in [0,T]$
  \begin{align*}
     \langle w(t),h \rangle + \gamma \int_0^t\langle w(s),h \rangle\;ds =-\int_0^t\langle u(s)\cdot &\nabla w(s),h \rangle\;ds + \langle x,h \rangle   + \int_0^t\langle \eta(s),h \rangle\;ds.   
  \end{align*}
Thus, $t \mapsto \langle w(t),h\rangle$ is continuous, and the above result holds for any $t\in[0,T]$. Thanks to the uniqueness of $w(t;x)$, this implies that for any $t\in[0,T]$, $w(t)=w(t;x)$ and 
$$\langle w_n(t)-w(t;x),h \rangle_{W^{1,\infty},W^{-1,1}} \rightarrow 0,\;\forall h \in H^1. $$
Note that $w_n(t)$ and $w(t)$ are bounded in $W^{1, \infty}$ by Theorem \ref{Theorem:w bounded}. Since $H^1$ is dense in $W^{-1,1}$, by some standard argument, we have for any $t \in [0,T]$ and any $h \in W^{-1,1}$ 
$$ \langle w_n(t)-w(t;x),h \rangle_{W^{1,\infty},W^{-1,1}} \rightarrow 0. $$
\end{proof}

We denote by $\mathcal{B}_b(W^{1,\infty},\mathcal{T}_{w\star} )$ the set of real bounded Borel functions under the weak$\star$ topology of $W^{1,\infty}$. Under our assumption about the noise, $w(i),\; i\in \mathbb{N}$ can be seen as a random dynamical system in the sense that
 $$w(i)=S(w({i-1}),\eta_i), \; i \in \mathbb{N}, $$ where $\eta_i(t)= \eta(t+i-1),\;t \in [0,1] $ are i.i.d. We define the operators $P_i, \; i\in\mathbb{N}_0:=\mathbb{N}\cup\{0\}$ as 
$$P_i\phi(x)= \mathbb{E}[\phi(w(i;x))],$$ for any $\phi \in \mathcal{B}_b(W^{1,\infty},\mathcal{T}_{w\star} ) $. For $x \in W^{1,\infty}$, the family of probability measures $\{P_i(x,\cdot)\}_{i\in \mathbb{N}_0}$ is defined as $$P_i(x,B)= \mathbb{P}(w(i;x)\in B),B\in \mathcal{B}(W^{1,\infty},\mathcal{T}_{w\star} ) .$$
We want to use Krylov-Bogoliubov's argument to show the existence of invariant measures. In order to do so, we need to show the Feller property of $P_i$. We have the following proposition:
\begin{proposition}\label{prop:Feller}
    The operator $P_i$ is sequentially weak$\star$ Feller in $W^{1,\infty}$, that is 
    \begin{equation}
        P_i:  SC_b(W^{1,\infty},\mathcal{T}_{w\star}) \rightarrow SC_b(W^{1,\infty} ,\mathcal{T}_{w\star}).    
    \end{equation} 
   Equivalently, one has
\begin{equation}
    P_i:C_b(W^{1,\infty} ,\mathcal{T}_{bw\star} ) \rightarrow C_b(W^{1,\infty} ,\mathcal{T}_{bw\star} ). 
    \end{equation}
\end{proposition}
\begin{proof}
    This is a direct consequence of Theorem \ref{Theorem: initial value continuous} and the Lebesgue dominated convergence theorem.
\end{proof}
Next, we want to show that $P_i$ defines a Markovian semigroup. Note that $W^{1, \infty}$ is not separable under strong topology. However, it is separable under weak$\star$ topology by our arguments in Section \ref{Sec:Pre}. Then, the Markovian property follows from a similar argument to the classical theory of Markovian semigroups in a Polish space.
\begin{proposition}
     Let $\gamma\geq 0$.
     Then, for every $\phi \in SC_b(W^{1,\infty},\mathcal{T}_{w\star} ) $, $x \in W^{1,\infty}$ and $i, j\in\mathbb{N}_0$, we have
     \begin{equation}\label{eqn: Markov 1}
         \mathbb{E}\left[\phi(w(i+j;x))|\mathcal{F}_i \right]= P_j\phi(w(i;x)),\; \mathbb{P}-a.s..
     \end{equation}
\end{proposition}
\begin{proof}
We denote by $w^x_{i,i+j}$ the solution of equation (\ref{eqn:Euler 1}) on the time interval $[i,i+j]$ evaluated at time $i+j$ and started from $x$ at time $i$. For $\phi \in SC_b(W^{1,\infty},\mathcal{T}_{w\star} )  $, $x\in W^{1,\infty}$ and $i,j\in \mathbb{N}$, to prove (\ref{eqn: Markov 1}) is equivalent to prove
    \begin{equation}\label{eqn:Markov 2}
        \mathbb{E}\left[\phi(w(i+j;x))\mathcal{Z} \right]= \mathbb{E}[P_j\phi(w(i;x))\mathcal{Z}],    
        \end{equation}
     for every bounded $\mathcal{F}_i$-measurable random variable $\mathcal{Z}$. By Theorem \ref{Theorem:w bounded}, we have $\mathbb{P}(w(i;x)\in W^{1,\infty})=1 $. Moreover, by the uniqueness of the solution, we have
     $$w(i+j;x)=w_{i,i+j}^{w(i;x)},\;\mathbb{P}-a.s.. $$ Then, to get (\ref{eqn:Markov 2}), it is sufficient to show that
     \begin{equation}\label{eqn:Markov 3}
         \mathbb{E}\left[\phi(w_{i,i+j}^\chi)\mathcal{Z} \right]= \mathbb{E}[P_j\phi(\chi)\mathcal{Z}],         \end{equation}
         for any $W^{1,\infty}$ valued $\mathcal{F}_i$-measurable random variable $\chi$. Since $W^{1,\infty}$ is separable under the $\mathcal{T}_{w\star}$ topology, for such a random variable $\chi$, there exists a sequence of $\mathcal{F}_i$-measurable $W^{1,\infty} $ valued random variables,
 $$\chi_n=\sum_{l=1}^{k_n}\chi_n^{(l)}1_{A_n^{(l) }}$$ with $\chi_n^{(l)} \in W^{1,\infty}$ and $A_n^{(l)} \in \mathcal{F}_i$ with $\{A_n^{(1)},...,A_n^{(k_n)} \}$ a partition of $\Omega$, such that $\chi_n$ converges $\mathbb{P}-a.s.$ under $\mathcal{T}_{w\star} $ topology to $\chi$. By Theorem \ref{Theorem: initial value continuous} and Proposition \ref{prop:Feller}, we have $w_{i,i+j}^{\chi_n}$ converges weakly$\star$ to $w_{i,i+j}^{\chi}$ and  $$ P_j\phi(\chi_n)\rightarrow P_j \phi(\chi),\;\mathbb{P}-a.s.. $$ Therefore, to complete the proof, it is sufficient to show that
 \begin{equation}\label{eqn:Markov 4}
     \mathbb{E}\left[\phi(w_{i,i+j}^{\chi_n})\mathcal{Z} \right]= \mathbb{E}[P_j\phi(\chi_n)\mathcal{Z}],  
     \end{equation}
for any $W^{1,\infty}$-valued $\mathcal F_i$-measurable random variable $\chi$. For every deterministic $\chi_n^{(l)}\in W^{1,\infty}$, we have $w^{\chi_n^{(l)}}_{i,i+j}$ is independent of $\mathcal{F}_i$. Hence, 
 \begin{align*}
     \mathbb{E}\left[\phi(w_{i,i+j}^{\chi_n^{(l)}})\mathcal{Z} \right]&= \mathbb{E}\left[\mathbb{E}\big[\phi(w_{i,i+j}^{\chi_n^{(l)}})\mathcal{Z} \mid \mathcal{F}_i\big]\right] 
     = \mathbb{E}\left[\mathbb{E}\big[\phi(w_{i,i+j}^{\chi_n^{(l)}}) \mid \mathcal{F}_i\big]\mathcal{Z}\right]\\
     &= \mathbb{E}\left[\mathbb{E}\big[\phi(w_{i,i+j}^{\chi_n^{(l)}}) \big]\mathcal{Z}\right] = \mathbb{E}\left[\phi(w_{i,i+j}^{\chi_n^{(l)}}) \right] \cdot \mathbb{E}[\mathcal{Z}]\\
     &= \mathbb{E}\left[\phi(w(j;\chi_n^{(l)}))\right] \cdot \mathbb{E}[\mathcal{Z}]= P_j\phi(\chi_n^{(l)})\cdot \mathbb{E}[\mathcal{Z}]\\
     &= \mathbb{E}\left[P_j\phi(\chi_n^{(l)})\mathcal{Z} \right].     \end{align*}
     Since this holds for every $\mathcal{F}_i$-measurable $W^{1,\infty} $ valued random variable $\mathcal{Z}$ and $A_n^{(l)}$ is $\mathcal{F}_i$ measurable, we have
     \begin{equation}\label{eqn:Markov 5}
     \mathbb{E}\left[\phi(w_{i,i+j}^{\chi_n^{(l)}})1_{A_n^{(l)}}\mathcal{Z} \right]= \mathbb{E}\left[P_i\phi(w(j;\chi_n^{(l)}))1_{A_n^{(l)} }\mathcal{Z} \right].     \end{equation}
     Note that the equation is solved pathwise; we have 
     \begin{equation}\label{eqn:Markov 6}
         w_{i,i+j}^{\chi_n}= \sum_{l=1}^{k_n}w_{i,i+j}^{\chi_n^{(l)}}1_{A_n^{(l)}},\; \phi(w_{i,i+j}^{\chi_n})= \sum_{l=1}^{k_n}\phi\left(w_{i,i+j}^{\chi_n^{(l)}}\right)1_{A_n^{(l)}}     \end{equation}
    and
    \begin{equation}\label{eqn:Markov 7}
        P_j\phi\left(\chi_n \right)= \sum_{l=1}^{k_n}P_j\phi\left(\chi_n^{(l)} \right)1_{A_n^{(l)}}.    \end{equation}
    By substituting (\ref{eqn:Markov 5}), (\ref{eqn:Markov 6}) and (\ref{eqn:Markov 7}) into (\ref{eqn:Markov 4}), we get the desired result.
 
\end{proof}

\begin{corollary}
    For any $i, j\in\mathbb{N}_0$, we have $P_{i+j}=P_iP_j$ on $SC_b(W^{1,\infty},\mathcal{T}_{w\star} ) $.
\end{corollary}
\begin{proof}
    This is a direct consequence of the above proposition. By taking the expectation on both sides of (\ref{eqn: Markov 1}), we have
 $$\mathbb{E}\left[\phi(w(i+j;x)) \right]= \mathbb{E}\left[P_j\phi(w(i;x))\right], $$ which implies $(P_{i+j}\phi)(x)= (P_i(P_j\phi))(x) $.
\end{proof}
According to the above arguments, $\{P_i\}_{j\in \mathbb{N}_0}$ acting on $SC_b(W^{1,\infty},\mathcal{T}_{w\star} )= C_b(W^{1,\infty} ,\mathcal{T}_{bw\star} ) $ is a Markovian semigroup. We say that a probability measure $\mu$ on $(W^{1,\infty},\mathcal{B}(\mathcal{T}_{bw\star} ))$ is an invariant measure for equation (\ref{eqn:Euler 1}) if 
\begin{equation}\label{eqn:invariant measure}
    \int P_i\phi\;d\mu=\int \phi \;d\mu, \forall i \in \mathbb{N}_0,\;\forall \phi \in C_b(W^{1,\infty} ,\mathcal{T}_{bw\star} ). \end{equation}
Similarly, we denote by $\{P_i^\epsilon\}_{t\geq 0}$ the associated semigroup to $w_\epsilon$ and by $\mu^\epsilon$ its invariant measure. The standard approach for proving the existence of invariant measures is Krylov–Bogoliubov’s method. This method is usually used on Polish space (see Section 3.1 of \cite{da1996ergodicity}). In \cite{bessaih2020invariant}, Bessaih and Ferrario dealt with $L^\infty$ space and used Krylov–Bogoliubov’s argument with weak$\star$ topology of $L^\infty$. We apply this method to $W^{1,\infty}$ space with weak$\star$ topology and prove the following existence result:
\begin{theorem}\label{Theorem:existence 1}
    There exists $\gamma^\star>0$ such that for any $\gamma \geq \gamma^\star $, the equation (\ref{eqn:Euler 1}) has at least one invariant measure.
\end{theorem}
\begin{proof}
    We denote by $m_i$ the law of random variable $w(i;0)$ on $\mathcal{B}(\mathcal{T}_{bw\star} ) $. We will show that the sequence of probability measures $\{\mu_n\}_{n\in \mathbb{N}}$ defined by $$\mu_n=\frac{1}{n}\sum_{i=1}^n m_i,$$ is $\mathcal{T}_{bw\star}$-tight. Actually, by the Banach–Alaoglu theorem, the set $\{\|x\Vert_{W^{1,\infty}}\leq R\}$ is $\mathcal{T}_{bw\star}$-compact. Thanks to the global boundedness result for $w$ in Theorem \ref{Theorem:w bounded}, it follows immediately that $\{\mu_n\}_{n\in \mathbb{N}}$ is $\mathcal{T}_{bw\star}$-tight; that is, for any $\varepsilon>0$, there exists a $\mathcal{T}_{bw\star}$-compact set $K_\varepsilon\subset W^{1,\infty}$ such that  $$\inf_n\mu_n(K_\varepsilon)>1-\varepsilon. $$

    Next, we apply an extended version of Prokhorov’s theorem for non-metric spaces. This version is due to Jakubowski (see Theorem 3 in \cite{jakubowski1998almost}). We also note that this extended version of Prokhorov’s theorem was used in \cite{bessaih2020invariant} in the $L^\infty$ setting. The crucial point here is that the space $W^{-1,1}$, viewed as the predual of $W^{1,\infty}$, is separable, similarly to the $L^1$ space. To apply Jakubowski’s theorem, we need to verify that $W^{1,\infty}$ is countably separated under the bounded weak$\star$ topology $\mathcal T_{bw^\star}$. Namely, we must show that there exists a countable family of $\mathcal T_{bw^\star}$-continuous functions $\{g_i:W^{1,\infty}\rightarrow [0,1] \}_{i \in \mathbb{N}} $ that separates points of $W^{1,\infty}$. Since $W^{-1,1}$ is separable and $W^{1,\infty}=(W^{-1,1})^*$, there exists a countable sequence $\{h_l\}\in W^{-1,1}$ separating points of $W^{1,\infty}$; that is, for any two elements $x\neq y $ in $W^{1,\infty}$ there exists $h_l$ such that $$\langle x,h_l\rangle_{W^{1,\infty},W^{-1,1}} \neq \langle y,h_l\rangle_{W^{1,\infty},W^{-1,1}} .$$ Since that mapping $g_l(\cdot)= \langle \cdot,h_l\rangle_{W^{1,\infty},W^{-1,1}} $ is $\mathcal{T}_{w\star} $ continuous, it is also $\mathcal{T}_{bw\star} $ continuous.

    By this extended Prokhorov’s theorem and the $\mathcal{T}_{bw\star}$-tightness of $\{\mu_n\}_{n\in \mathbb{N}}$, there exists a subsequence $\{\mu_{n_k}\}_{k\in \mathbb{N}}$, $n_k \uparrow \infty$ and a probability measure $\mu$ on $\mathcal{B}(\mathcal{T}_{bw\star} ) $ such that $\mu_{n_k} $ converges weakly to $\mu$ as $k\rightarrow \infty$, that is, 
    \begin{equation}\label{limit:weak convergence}
        \int\phi\;d\mu_{n_k}\rightarrow \int\phi\;d\mu, \forall\phi \in C_b(W^{1,\infty} ,\mathcal{T}_{bw\star} ).    \end{equation}
     We have shown in Proposition \ref{prop:Feller} that $P_t\phi \in C_b(W^{1,\infty} ,\mathcal{T}_{bw\star} ) $ for $\phi \in C_b(W^{1,\infty} ,\mathcal{T}_{bw\star} ) $. Then, for $t\geq 0$,
    \begin{align*}
        \langle P_i\phi, \mu\rangle&= \lim_{k\rightarrow \infty}\langle P_i\phi, \mu_{n_k} \rangle\\&= \lim_{k\rightarrow \infty}\frac{1}{n_k}\langle P_i\phi,\sum_{j=1}^{n_k} m_{j} \rangle \\&=
        \lim_{k\rightarrow \infty}\frac{1}{n_k}\langle \phi,\sum_{j=i+1}^{i+n_k} m_{j} \rangle \\&= 
        \lim_{k\rightarrow \infty}\Big[\frac{1}{n_k}\langle \phi,\sum_{j=1}^{n_k} m_{j} \rangle + \frac{1}{n_k}\langle \phi,\sum_{j=n_k+1}^{i+n_k} m_{j} \rangle- \frac{1}{n_k}\langle \phi,\sum_{j=1}^{i} m_{j} \rangle \Big]\\&=\lim_{k\rightarrow \infty} \langle \phi, \mu_{n_k} \rangle = \langle \phi, \mu \rangle.       
        \end{align*}
        This completes the proof.

\end{proof}

\begin{theorem}\label{Theorem: existence 2}
    Let $\gamma>0$. There exists $\epsilon(\gamma)>0$ such that for any $0<\epsilon<\epsilon(\gamma)$, the family $\{P_i^\epsilon\}_{i \in \mathbb{N}_0}$ for equation (\ref{eqn:Euler 2}) admits at least one invariant measure. 

    \end{theorem}
\begin{proof}
    We have shown in Corollary \ref{corollary:boundedness} that for $2<p\leq\infty$
 $$\sup_{t\geq0}\|\nabla w_\epsilon(t)\Vert_{p} < \infty,\; \mathbb{P}-a.s..  $$
 Then, the existence of invariant measures $\mu^\epsilon$ follows from the same arguments as Theorem \ref{Theorem:existence 1}.
\end{proof}

\begin{remark}\label{remark:extension}
    In \cite{bessaih2020invariant}, the authors studied damped stochastic Euler equations perturbed by Gaussian noise, with the damping term $\gamma>0$. They proved the existence of invariant measures on the weak$\star$ topology of $L^\infty$.   
   Our result for the existence of invariant measures is stronger than the result in \cite{bessaih2020invariant}. Note that the weak$\star$ topology of $W^{1,\infty}$ is stronger than the weak$\star$ topology of $L^\infty$. In fact, $\mu$ and $\mu^\epsilon$ in Theorem \ref{Theorem:existence 1} and Theorem \ref{Theorem: existence 2} are naturally viewed as probability measures on $E$ by taking $\bar{\mu}(B)=\mu(B\cap W^{1,\infty}),\;\bar{\mu}^\epsilon(B)=\mu^\epsilon(B\cap W^{1,\infty}) $ for any $B \in \mathcal{B}(E)$. These extensions are well-defined since the inclusion $i:(W^{1,\infty},\mathcal{T}_{bw\star})\mapsto E$ is continuous. 
   Since $E$ is a Polish space, the extension of the transition probabilities and the preservation of the Markov property are standard. Then, by (\ref{C(E)}), they satisfy $$\int_{E} P_i\phi\;d\bar{\mu} = \int_{W^{1,\infty}} P_i\phi\;d\mu  =\int_{W^{1,\infty}} \phi \;d\mu = \int_{E} \phi \;d\bar{\mu}, \forall i \in \mathbb{N}_0,\;\forall \phi \in C_b(E), $$ 
 while in \cite{bessaih2020invariant}, the test function $\phi$ was only proved to be valid for continuous bounded functions on $L^{\infty}$ space with weak$\star$ topology.
\end{remark}

\section{Exponential mixing}
In this section, we prove the uniqueness of the invariant measure and establish exponential convergence to equilibrium. More precisely, we prove exponential mixing for the discrete-time Markov
semigroup $(P_i)_{i\in\mathbb N_0}$ associated with the dynamics. We have the following theorem.
\begin{theorem}
    There exists $\bar\gamma>0$ such that, for every $\gamma\ge\bar\gamma$,
equation \eqref{eqn:Euler 1} admits a unique invariant measure
$\mu$ on $E$. Moreover, for any $x \in E$ and $\phi \in \text{Lip}_b(E)$,
    \begin{equation}\label{eqn:exponential mixing 1}
        |P_i\phi(x)-\int_{E}\phi\;d\mu\vert \leq c(x)e^{-c(\gamma)i}\|\phi\Vert_{\text{Lip}},
    \end{equation}
    where $c(x)$, $c(\gamma)>0$ depend on $x$ and $\gamma$ respectively.
\end{theorem}
\begin{proof}
 Let $x,y\in E$, and denote by $w^x$ the solution of
(\ref{eqn:Euler 1}) with initial condition $x$. Set
$\rho^{x-y}=w^x-w^y$. We fix once and for all a reference point
$z\in W^{1,\infty}$, for instance $z=0$.
 Thanks to (\ref{eqn:Y}), we obtain by subtracting the equations satisfied by $w^z(t)$ and $w^y(t)$, multiplying $\rho^{z-y}(t)\cdot |\rho^{z-y}(t) \vert^{p-2} $ with $p>2$, and then integrating over $\mathbb{T}^2$:
    \begin{align*}
       \frac1p \partial_t \|\rho^{z-y}\Vert_p^p + \gamma \|\rho^{z-y}(t)\Vert_p^p &= - \langle (\mathcal{K}\ast\rho^{z-y})(t)\cdot \nabla w^z(t), \rho^{z-y}(t)\cdot |\rho^{z-y}(t) \vert^{p-2}\rangle \\&\ \ -\langle (\mathcal{K}\ast w^y)(t)\cdot \nabla \rho^{z-y}(t), \rho^{z-y}(t)\cdot |\rho^{z-y}(t) \vert^{p-2}\rangle.       \end{align*}
 Since $\nabla \cdot (\mathcal{K}\ast w^y) (t)=0$, we have \begin{align*}
     \langle (\mathcal{K}\ast w^y) (t)\cdot \nabla \rho^{z-y}(t), \rho^{z-y}(t)\cdot |\rho^{z-y}(t) \vert^{p-2}\rangle &= \frac1p \langle (\mathcal{K}\ast w^y) (t),\nabla | \rho^{z-y} (t)\vert^p \rangle\\&= -\frac1p\langle \nabla \cdot (\mathcal{K}\ast w^y) (t),|\rho^{z-y}(t)\vert^p \rangle=0. \end{align*}
     Then, by Lemma \ref{Lemma:K Lp} and H{\"o}lder's inequality, we have
     $$\frac1p \partial_t \|\rho^{z-y}\Vert_p^p + \gamma \|\rho^{z-y}(t)\Vert_p^p \leq C\|\nabla w^z(t)\Vert_\infty \cdot \|\rho^{z-y}(t)\Vert_p^p.  $$
     Since $\partial_t\| \rho^{z-y}\Vert_{p}^p= p \|\rho^{z-y}(t)\Vert_{p}^{p-1}\cdot \partial_t\| \rho^{z-y}\Vert_{p}  $, we deduce that  
     $$\partial_t\| \rho^{z-y}\Vert_{p} \leq (-\gamma +C \|\nabla w^z(t)\Vert_\infty )\cdot \| \rho^{z-y}(t)\Vert_{p}.  $$
     Thus, by Gr{\"o}nwall inequality, for any $p>2$
     $$\| \rho^{z-y}(t)\Vert_{p} \leq \|z-y\Vert_p\cdot e^{\int_0^t(-\gamma+ C||\nabla w^z(s)\Vert_\infty )\;ds}. $$
Taking $p \rightarrow \infty$, the above inequality implies that
\begin{equation}
    \|w^z(t)-w^y(t)\Vert_\infty \leq \|z-y\Vert_\infty\cdot e^{\int_0^t(-\gamma+ C\|\nabla w^z(s)\Vert_\infty )\;ds}.\end{equation}
Since \(z\in W^{1,\infty}\) is fixed, Theorem~\ref{Theorem:w bounded} implies that, for
\(\gamma\ge \gamma^\star(z)\),
\[
M_z:=\sup_{t\ge0}\|\nabla w^z(t)\|_\infty<\infty .
\]
Choosing \(\gamma\) sufficiently large so that \(\gamma>C M_z\), we obtain
\[
\|w^z(t)-w^y(t)\|_\infty
\le \|z-y\|_\infty e^{-c(\gamma)t},
\qquad t\ge0,\quad \mathbb P\text{-a.s.}.
\]
By the same arguments as above, 
\begin{equation}
    \|w^z(t)-w^x(t)\Vert_\infty \leq C\|z-x\Vert_\infty  \cdot e^{-c(\gamma)t},\; \mathbb{P}-a.s..\end{equation}
Therefore, we have $\mathbb{P}-a.s. $
\begin{equation}\label{inequality: exponential convergence}
\begin{split}
     \|w^x(t)-w^y(t)\Vert_\infty &\leq \|w^z(t)-w^y(t)\Vert_\infty  + \|w^z(t)-w^x(t)\Vert_\infty \\  &\leq   C(\|z-y\Vert_\infty+ \|z-x\Vert_\infty ) \cdot e^{-c(\gamma)t}.
     \end{split}   
\end{equation}

Let $\mu$ be an invariant measure on $E$. That is to say
$$\int_{E}P_i\phi\;d\mu= \int_{E}\phi\;d\mu, \forall \phi \in C_b(E) .$$ The existence of $\mu$ has been shown by Theorem \ref{Theorem:existence 1} and Remark \ref{remark:extension}.
By (\ref{inequality: exponential convergence}), for $\phi \in \text{Lip}_b(E)$.
\begin{align*}
    \left|P_i\phi (x)-\int_{E}\phi\;d\mu\right\vert&= \left| \mathbb{E}[\phi (w^x(i))]- \int_{E} \phi \;d\mu \right\vert\\&= \left| \mathbb{E}[\phi(w^x(i))]- \int_{E} P_i\phi \;d\mu \right\vert\\&=    \left| \int_{E} \mathbb{E}\left[ \phi (w^x(i))-  \phi(w^y(i)) \right ] \;\mu(dy)  \right\vert\\ &\leq C e^{-c(\gamma)i}\|\phi\Vert_{\text{Lip}}\int_{E}\left(\|z-y\Vert_\infty+ \|z-x\Vert_\infty \right)\;\mu(dy)\\& \leq c(x)e^{-c(\gamma)i}\|\phi\Vert_{\text{Lip}} .
    \end{align*}
We thus get (\ref{eqn:exponential mixing 1}). Next,we will show the uniqueness of $\mu$. 
We denote by $\bar{P}_i(x,\cdot)$ the natural extension for probability measures $P_i(x,\cdot)$ on $E$ as in Remark \ref{remark:extension}. Then, by (\ref{eqn:exponential mixing 1})
\begin{equation}
      \|\bar{P}_i(x,\cdot)-\mu(\cdot)\Vert_{FM} \leq c(x)e^{-c(\gamma)i}.  
\end{equation}
Here, $\| \cdot\Vert_{FM}$ is the Fortet-Mourier norm defined as $$\|\nu\Vert_{FM}= \sup\{|\langle\phi,\nu\rangle\vert: \phi\in\text{Lip}_b(L^\infty), \|\phi\|_0\leq 1, \|\phi\Vert_{\text{Lip}}\leq 1\}.  $$ This implies the uniqueness. 
\end{proof}

Thanks to Corollary \ref{corollary:boundedness}, by similar arguments to those in the above theorem, we also have
\begin{theorem}
    Let $\gamma>0$. There exists $\epsilon(\gamma)>0$ such that there exists a unique invariant measure $\mu^\epsilon$ on $E$ for Eq. (\ref{eqn:Euler 2}). Moreover, for any $0<\epsilon<\epsilon(\gamma)$, $x \in E$ and $\phi \in \text{Lip}_b(E)$,
    \begin{equation}\label{eqn:exponential mixing 2}
        |P_i^\epsilon\phi(x)-\int_{E}\phi\;d\mu^\epsilon\vert \leq c(x)e^{-c(\gamma)i}\|\phi\Vert_{\text{Lip}},
    \end{equation}
    where $c(x)$, $c(\gamma)>0$ depend on $x$ and $\gamma$ respectively.
\end{theorem}

\begin{remark}
Notice that the extended measure $\bar{\mu}$ is an invariant measure on $E $. By the uniqueness result we proved above, $\bar{\mu}$ is exactly the unique invariant measure for the equation (\ref{eqn:Euler 1}). Moreover, by the construction of $\bar{\mu}$, its support is, in fact, on $W^{1,\infty}$. The same result also holds for $\mu^\epsilon$.
\end{remark}


\bibliographystyle{siam}
\newpage
\footnotesize
\bibliography{Ref}
\end{document}